\documentclass[10pt]{article} 

\usepackage[colorlinks=true, allcolors=blue]{hyperref}

\usepackage[utf8]{inputenc} 

\usepackage{geometry} 
\usepackage{graphicx} 

\usepackage[round]{natbib}

\usepackage{booktabs} 
\usepackage{array} 
\usepackage{paralist} 
\usepackage{verbatim} 
\usepackage{subfig} 

\usepackage{amsmath}
\usepackage{amsfonts}
\usepackage{amssymb}
\usepackage{dsfont}
\usepackage{amsthm}

\usepackage{multirow}

\allowdisplaybreaks

\theoremstyle{theorem}

\newtheorem{thm}{Theorem}
\newtheorem{lem}{Lemma}
\newtheorem{cor}{Corollary}
\theoremstyle{definition}
\newtheorem{defn}{Definition}

\newcommand{\RRR}{\mathbb{R}}

\newcommand{\EEE}{\mathbb{E}}
\newcommand{\PPP}{\mathbb{P}}
\newcommand{\dd}{\mathrm{d}}
\newcommand{\one}{\mathds{1}}

\newcommand{\twCRPS}{\mathrm{twCRPS}}
\newcommand{\CRPS}{\mathrm{CRPS}}
\newcommand{\survCRPS}{\mathrm{survCRPS}}
\newcommand{\LogS}{\mathrm{LogS}}
\newcommand{\twLogS}{\mathrm{twLogS}}

\newcommand{\QL}{\mathrm{QL}}
\newcommand{\IS}{\mathrm{IS}}
\newcommand{\LinS}{\mathrm{LinS}}
\newcommand{\twQL}{\mathrm{twQL}}
\newcommand{\twIS}{\mathrm{twIS}}
\newcommand{\ES}{\mathrm{ES}}

\usepackage{fancyhdr} 
\usepackage{sectsty}
\allsectionsfont{\sffamily\mdseries\upshape} 

\usepackage[nottoc,notlof,notlot]{tocbibind} 
\usepackage[titles,subfigure]{tocloft} 

\title{Evaluating time-to-event forecasts under fixed-horizon right-censoring}

\author{ Robert J. Taggart\thanks{Corresponding author: robert.taggart@bom.gov.au}, \,\, Nicholas Loveday, \,\, Simon Louis\\[1ex] Australian Bureau of Meteorology }

\date{15 September 2026}

\begin{document}

\maketitle

\begin{abstract}
\noindent Time-to-event forecasts frequently arise in applications with finite prediction or decision horizons, including in operational meteorology, hydrology, warranty analysis and insurance. In such settings, forecasts, observations, or both may only be available through their right-censored values relative to a common fixed censoring time. This article develops a framework for evaluating such forecasts.

For probabilistic forecasts, recent localization results for proper scoring rules imply coherent evaluation methods based on threshold-weighted versions of familiar scores, including the continuous ranked probability score (CRPS) and logarithmic score. These methods remain valid when forecasts themselves are right-censored to the prediction horizon.

For single-valued and interval-valued forecasts, we develop corresponding notions of local consistency and local elicitability under fixed-horizon right-censoring. We show that quantiles and prediction intervals whose endpoints are quantiles remain locally elicitable through threshold-weighted variants of quantile and interval scores. In contrast, the expectation functional is not locally elicitable, implying that mean-valued time-to-event forecasts do not admit analogous theoretically justified evaluation methods.

Several familiar diagnostic tools extend naturally to this setting. Threshold-weighted CRPS and quantile scores retain useful diagnostic representations through Brier-score decompositions and Murphy diagrams, while reliability diagrams based on quantile isotonic regression support diagnosis of conditional biases and forecast recalibration. A synthetic forecasting experiment together with operational flood and wind forecasting applications illustrates the methodology.

\vspace{10pt}

\end{abstract}

\section{Introduction}

Time-to-event forecasts are critical in domains where the timing of an event, not merely its occurrence, has major implications for decision making, planning, and resource allocation. Such forecasts play important roles across diverse fields, including healthcare, engineering, emergency management, human resources, finance, and the physical and life sciences. Examples include survival time for cancer patients \citep{llobera2000terminal}, failure time for mechanical devices \citep{wang2015failure}, timing of geo-mechanical failure and landslides \citep{mufundirwa2010new}, tsunami arrival time \citep{greenslade2019evaluation}, flood peak arrival \citep{langridge2021dynamic}, creep rupture time of polymers \citep{spathis2012creep}, customer churn \citep{huang2012customer}, time to loan default \citep{sarlija2009comparison}, time-to-failure forecasts used in warranty analysis \citep{babakmehr2024data}, and time-to-claim forecasting problems in insurance \citep{berestizhevsky2021cox}.

Given their importance, such forecasts should be evaluated using rigorous methods to identify superior forecasters or models, promote continuous improvement and ultimately support better decisions. While a robust suite of evaluation techniques and diagnostic tools has emerged over the past two decades (e.g., \citealp{gneiting2007strictly, gneiting2011making, ehm2016quantiles, dimitriadis2024evaluating}), two challenges commonly arise in the evaluation of time-to-event forecasts, requiring extensions of existing tools. First, at the evaluation time, some cases may lack realized event times because the event has not yet occurred. Second, some forecasting systems or services have finite prediction horizons, and the event may not be forecast within that horizon for some forecast cases.

The second challenge frequently arises in meteorological and hydrological applications. For example, consider an operational flood forecasting system that is run daily and predicts the river height at hourly intervals out to 168 hours (7 days). An event is the first time, from the model initialization time, that the river height exceeds the major flood threshold. Let $x$ denote the forecast event time and $t$ the realized event time, both measured in hours relative to model initialization. Each model run therefore generates a forecast case, with a common forecast horizon $\tau$ of 168 hours. When $x$ and $t$ are both less than $\tau$, standard forecast evaluation methods can be applied directly. However, the river may not be forecast to exceed the threshold within the forecast horizon, in which case it is only known that $x > \tau$. This motivates the use of the right-censored forecast $\min(x,\tau)$, which is known, in place of $x$, together with suitably modified evaluation methods. Here $\tau$ is the censoring time. Right-censoring is a familiar statistical concept in survival analysis, reliability engineering and related disciplines.

More generally, information about forecasts, observations, or both may only be known through their right-censored values. In this paper we focus on the special case where the censoring time $\tau$ is fixed rather than a random variable. This setting arises naturally in physics-based and machine-learning weather prediction systems and in hydrological forecasting driven by such systems. It also arises in operational forecast services that issue predictions over a fixed decision horizon, and time-to-claim or time-to-failure forecasts associated with finite-horizon insurance and warranty policies. In contrast, we do not consider the random censoring mechanisms that dominate much of the biomedical survival-analysis literature, where censoring typically arises from irregular or incomplete observation records.

The aim of this paper is to provide a coherent framework for evaluating time-to-event forecasts in this setting. We consider both probabilistic forecasts and commonly used forecast summaries such as the mean, median, quantiles, and prediction intervals. Our goal is to identify sound methods that permit coherent forecast ranking, together with diagnostic tools that help explain forecast performance and reveal conditional biases. Wherever possible, we draw on existing theory and tools from forecast verification and survival analysis. Where existing results are unavailable, we develop the necessary theory. Our exposition first reviews evaluation when forecasts and observations are fully known, then examines modifications required in the presence of right-censored observations, and finally shows how the resulting methods naturally accommodate right-censored forecasts.

To assess and compare the accuracy of different forecasting methods, we evaluate forecasts against the corresponding observed outcomes. For probabilistic forecasts, including those from ensemble prediction systems \citep{mureau1993ensemble, molteni1996ecmwf, schaake2007hepex, cloke2009ensemble}, this is typically done using scoring rules, which assign a numerical score $S(F,t)$ to a predictive distribution $F$ when compared with the corresponding realization $t$. A well-developed theory of proper scoring rules has emerged over the past two decades \citep{gneiting2007strictly, gneiting2014probabilistic}. Proper scoring rules incentivize honest forecasting and provide a principled basis for ranking competing forecasting systems.

However, the standard theory of proper scoring rules assumes evaluation against the realized event time $t$. When only the right-censored realization $\min(t,\tau)$ is available, existing evaluation methods require modification. Recent work by \citet{de2025localizing} provides a general framework for localizing strictly proper scoring rules to regions of interest. In the fixed-horizon setting considered here, these results yield natural evaluation methods based on threshold-weighted versions of familiar scoring rules, including the continuous ranked probability score (CRPS) and logarithmic score \citep{gneiting2011comparing, diks2011likelihood}. We present these methods in a form tailored to the evaluation of time-to-event forecasts and illustrate their use through synthetic and operational forecasting examples.

Many forecast products are communicated not as full predictive distributions but as summaries of those distributions known as functionals. Examples include expected event times, median event times, upper quantiles, and prediction intervals describing plausible event times. The evaluation of such forecasts is commonly studied through the theory of elicitable functionals and consistent scoring functions \citep{gneiting2011making, lambert2008eliciting}. While many functionals are known to be elicitable when realizations are fully observed, corresponding questions under fixed-horizon right-censoring (or within the more general localizing framework of \citealt{de2025localizing}) appear not to have been studied previously.

To address this gap, we develop corresponding notions of local consistency and local elicitability for time-to-event forecast functionals under right-censoring. We show that quantile functionals remain elicitable in this setting through threshold-weighted variants of quantile loss \citep{taggart2022point}, and that prediction intervals whose endpoints are quantiles, including the interquartile interval, likewise remain elicitable through threshold-weighted interval scores. In contrast, we establish general conditions under which elicitability fails under right-censored evaluation, and use these results to show that the mean (or  expectation) functional is not elicitable in this setting. As a result, expected time-to-event forecasts, including those based on ensemble means, do not admit the same type of theoretically justified evaluation methods that are available for quantiles and prediction intervals.

Beyond forecast ranking, we demonstrate that several familiar diagnostic tools extend naturally to the fixed-horizon setting. Threshold-weighted CRPS admits a Brier-score decomposition  \citep{schumacher2003assess} that supports graphical diagnostics across decision thresholds, while threshold-weighted quantile scores admit Murphy-diagram representations \citep{ehm2016quantiles}. We also show that reliability diagrams based on quantile isotonic regression \citep{gneiting2023regression} remain useful for diagnosing conditional miscalibration and recalibrating time-to-event forecasts when both forecasts and realizations are right-censored.

The methodology is illustrated using a synthetic forecasting experiment together with two operational applications from the Australian Bureau of Meteorology: first passage time forecasts for river flooding and time-to-strong-wind forecasts. These examples demonstrate that the proposed methods remain practical and informative in forecasting systems where right-censoring is a natural consequence of finite forecast horizons.

The remainder of the article is organised as follows. Section~\ref{s:evaluation when known} reviews proper scoring rules, consistent scoring functions, and elicitability for uncensored time-to-event forecasts and introduces a synthetic forecasting experiment.  Section~\ref{s:right censored evaluation} presents evaluation methods for right-censored realizations and forecasts. Section~\ref{s:real examples} presents applications to operational flood and wind forecasting systems, with an emphasis on rankings, diagnostics and recalibration. Section~\ref{s:discussion} concludes with a summary of implications and key results for practitioners. Mathematical proofs are given in the appendix.

\section{Evaluation of time-to-event forecasts when both forecasts and realizations are known}\label{s:evaluation when known}

This section reviews established evaluation methods for time-to-event forecasts when both forecasts and realizations are known. We focus on the key concepts of strictly proper scoring rules, strictly consistent scoring functions, and elicitability, as these form the foundation for the modifications introduced later to handle right-censored data. To illustrate these concepts, and new concepts introduced later, we also present a synthetic experiment for time-to-event forecasts.

The general mathematical set-up and notation is as follows. The set of possible times at which a well-defined event could occur will be represented by an interval $I$ of the real line taking the form $[0,\infty)$, $[0,b)$ or $[0,b]$ for some positive $b$. The interior of $I$ consists of $I$ excluding its endpoints. The time-to-event random variable will typically be denoted by $T$ and takes values in $I$, and its realization (or observation) will typically be denoted by $t$. (For now we do not consider the case where there is positive probability that the event never materializes, though an extension of the theory to this case is discussed in Section~\ref{ss:right-censored forecasts}.)

We adopt a measure-theoretic framework for probability distributions and probability density functions (PDFs). To make the body of this paper as accessible as possible, technical terms related to measure theory are restricted to this paragraph and the appendix. Suffice to say, the setting described below is very general and includes virtually every probability distribution and PDF that arises in practical applications. Throughout, a time-to-event predictive distribution (i.e. probabilistic forecast) $F$ is assumed to be a Borel probability measure on $I$, and a class $\mathcal{F}$ of time-to-event predictive distributions is assumed to be contained within the set of Borel probability measures on $I$. A predictive distribution $F$ will be identified with its cumulative distribution function (CDF), such that given any time $s$ in $I$, $F(s)$ is the forecast probability that $T$ does not exceed $s$. Throughout we tacitly assume that scoring functions $S$ and associated weight functions $w$ are Borel measurable. Occasionally we work with classes $\mathcal{F}$ of probability measures that have densities (i.e. PDFs); in this case each probability measure in $\mathcal{F}$ is absolutely continuous with respect to Lebesgue measure unless stated otherwise.

\subsection{Proper scoring rules}

A scoring rule $S$ assigns a numerical score $S(F, t)$ to each pair $(F, t)$, where $F$ is a time-to-event predictive distribution and $t$ is the time-to-event realization. In this paper we adopt the convention that scoring rules are negatively oriented, so that the score represents a penalty and a lower score is better.

Given a random variable $T$ and probability distributions $F$ and $G$, we write $\EEE_F S(G,T)$ to denote the expectation of the score $S(G,T)$ given that the random variable $T$ is distributed according to $F$, assuming tacitly that the expectation exists.

\begin{defn}
Suppose that $\mathcal{F}$ is a class of probability distributions on $I$.
A scoring rule $S:\mathcal{F}\times I \to(-\infty,\infty]$ is said to be \textit{proper} relative to the class $\mathcal{F}$ if
\begin{equation}\label{eq:proper scoring rule}
\EEE_F S(F,T) \leq \EEE_F S(G,T)
\end{equation}
for all $F$ and $G$ in $\mathcal{F}$. It is  \textit{strictly proper} if equality in (\ref{eq:proper scoring rule}) implies that $F=G$.
\end{defn}

Proper scoring rules offer several advantages. If a time-to-event process $T$ follows distribution $F$, then $F$ achieves the best possible expected score under a proper scoring rule, and uniquely so under a strictly proper scoring rule. Moreover, if forecaster A makes optimal use of their available information and forecaster B has access to additional information and uses it optimally, then forecaster B will attain a better expected score when the scoring rule is strictly proper \citep{holzmann2014role}. Finally, if a forecaster genuinely believes the time-to-event process has distribution $F$, then issuing $F$ optimizes their expected score under a proper scoring rule. Thus, score optimization aligns with honest forecasting, avoiding the forecaster’s dilemma \citep{lerch2017forecaster}.

We give two examples of scoring rules that are proper relative to classes $\mathcal{F}$ defined on the interval $I$ and one that is not. The continuous ranked probability score $\CRPS$ is defined by
\begin{align}
\CRPS(F,t) 
&= \int_I \left( \one\{s\geq t\}-F(s)\right)^2\,\dd s \label{eq:CRPS1} \\
&= \EEE_F |T - t| - \tfrac{1}{2} \EEE_F |T - T'| \label{eq:CRPS2},
\end{align}
where $\one$ denotes the indicator function so that $\one\{s\geq t\}$ takes the value $1$ if $s\geq t$ and $0$ otherwise, and where in the last term $T$ and $T'$ are independent random variables with CDF $F$. The $\CRPS$ is strictly proper relative to the class $\mathcal{F}$ of probability measures with finite first moment \citep{matheson1976scoring, gneiting2007strictly}. The logarithmic score $\LogS$, defined by
\begin{equation}\label{eq:LogS}
\LogS(F,t) = -\ln F'(t),
\end{equation}
is strictly proper relative to the class $\mathcal{F}$ of probability distributions whose members $F$ have a probability density function (PDF) $F'$ \citep{good1952rational, gneiting2007strictly}, and is closely related to maximum likelihood estimation. In contrast, the linear score $\LinS$, given by
\[
\LinS(F,t) = -F'(t)
\]
for predictive distributions $F$ with a probability density function $F'$, is not a proper scoring rule, but rewards overprediction at the modes of a forecaster’s true predictive density \citep{winkler1969scoring, gneiting2007strictly}.

We illustrate these properties using a simple synthetic experiment. Let $\mathrm{Gamma}(\alpha, \beta)$ denote the gamma distribution positive shape parameter $\alpha$, positive rate parameter $\beta$, and support on $[0,\infty)$. To set up the experiment, suppose that for each forecast case $i$, where $1\leq i\leq 10000$, the time-to-failure $T_i$ of the $i$th product from an assembly line has distribution given by $T_i=X_i+Y_i+Z_i$, where $X_i\sim \mathrm{Gamma}(3, 1)$, $Y_i\sim \mathrm{Gamma}(2, 1)$, $Z_i \sim \mathrm{Gamma}(1, 1)$, the random variables are pairwise independent, and the units are in months. The gamma distribution has the property that $W_1+W_2\sim\mathrm{Gamma}(\alpha_1+\alpha_2, \beta)$ whenever $W_j\sim\mathrm{Gamma}(\alpha_j, \beta)$. It follows that the marginal distribution of $T_i$ is $\mathrm{Gamma}(6,1)$. The PDF of the $\mathrm{Gamma}(6,1)$ distribution is shown by the red curve in Figure~\ref{fig:synthetic example}a. In the following, denote by $a + \mathrm{Gamma}(\alpha,\beta)$ the translated gamma distribution with support on $[a,\infty)$, and by $x_i$, $y_i$, $z_i$ and $t_i$ the realizations of the random variables $X_i$, $Y_i$, $Z_i$ and $T_i$. We consider five forecasters, each of whom issues a predictive distribution $F_i$ for the random variable $T_i$ as follows.
\begin{enumerate}
\item \textit{LowInfoLucy} only has knowledge of the marginal distribution of $T_i$. For each forecast case she issues the ideal forecast based on this information, namely $F_i=\mathrm{Gamma}(6, 1)$.
\item \textit{ModInfoMuli} knows the realization $x_i$ and the marginal distribution of $Y_i+Z_i$. He issues the ideal forecast based on this information, namely $F_i=x_i+\mathrm{Gamma}(3,1)$.
\item \textit{HighInfoHannah} knows the combined realization $x_i+y_i$ and the true distribution of $Z_i$. She issues the ideal forecast based on this information, namely $F_i=x_i+y_i+\mathrm{Gamma}(1,1)$.
\item \textit{PessimisticPenny} has the same knowledge as HighInfoHannah but underpredicts $T_i$ by misspecifying the rate parameter of $Z_i$. She issues the forecast $F_i=x_i+y_i+\mathrm{Gamma}(1,2)$.
\item \textit{OptimisticOmar} has the same knowledge as HighInfoHannah but overpredicts $T_i$ by misspecifying the rate parameter of $Z_i$. He issues the forecast $F_i=x_i+y_i+\mathrm{Gamma}(1,1/3)$.
\end{enumerate}

\begin{figure}[t!]
\centering
\includegraphics[width=0.75\textwidth]{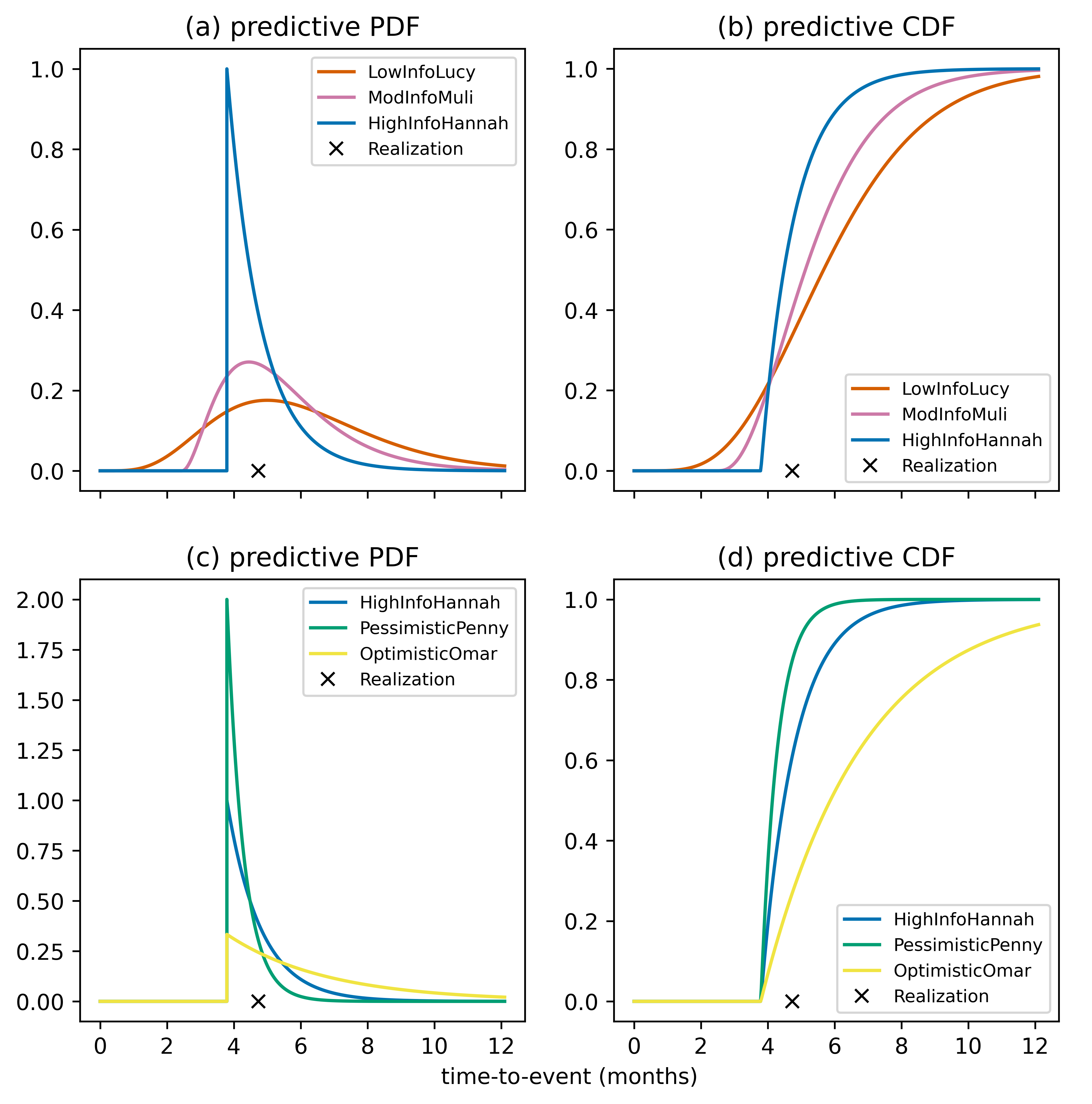}
\caption{Example of time-to-event predictive PDFs and CDFs of the five forecasters in the synthetic experiment, along with the corresponding realization. The top row illustrates the distributions of the forecasters who make ideal forecasts, but with access to different information. The bottom row shows the distributions of the forecasters who have access to the same information, but where two of the forecasters' distributions are misspecified.}
\label{fig:synthetic example}
\end{figure}

Figure~\ref{fig:synthetic example} illustrates the predictive PDFs and CDFs for each forecaster for one particular forecast case where $x_i=2.45$, $y_i=1.34$, $z_i=0.94$ and $t_i=4.53$. Panel (a) shows the PDFs of the three forecasters who issue ideal predictions based on the information that is available to them, and illustrates the effect of having access to increasing information: the predictions become sharper. The CDFs in panel (d) show the effect of misspecification: PessimisticPenny's CDF lies above that of HighInfoHannah, indicating systematically earlier predicted failure times, whereas OptimisticOmar's CDF lies below it, indicating systematically later predicted failure times. The red distributions issued by LowInfoLucy do not change with forecast case $i$ and equal the marginal distribution of the time-to-failure process.

The experiment is designed so that expected pairwise rankings of predictive performance among some forecasters are known \textit{a priori} as follows:
\begin{equation}\label{eq:rankings}
\begin{cases}
\text{ModInfoMuli's forecasts are superior to LowInfoLucy's forecasts;} \\
\text{HighInfoHannah's forecasts are superior to those of any other forecaster.}
\end{cases}
\end{equation}
This is because (i) LowInfoLucy, ModInfoMuli and HighInfoHannah issue ideal forecasts with respect to increasing nested information sets, and (ii) HighInfoHannah, PessimisticPenny and OptimisticOmar share the same information set but only HighInfoHannah specifies forecasts correctly. Consequently, given a sufficiently large and representative sample of forecast cases, any sound scoring rule should rank forecasters consistently with (\ref{eq:rankings}). 

Table~\ref{tab:scoring rule results} gives the mean scores for each forecaster in the synthetic experiment using $\CRPS$, $\LogS$ and $\LinS$. The CRPS for a gamma distribution can be efficiently calculated using the formula of \citet[Section~3.2]{scheuerer2015probabilistic}. A lower mean score is intended to convey that predictive performance is better on average. Note that rankings from the two strictly proper scoring rules $\CRPS$ and $\LogS$ are consistent with (\ref{eq:rankings}). The lowest mean score as measured by $\LinS$ was attained by PessimisticPenny, violating (\ref{eq:rankings}), and illustrating that ranking forecasters using a non-proper scoring rule can lead to misguided inferences. Given a scoring rule, all pairwise performance rankings in Table~\ref{tab:scoring rule results} are statistically significant in the following sense: the null hypothesis of equal predictive performance is rejected at the 95\% confidence level using the \citet{hering2011comparing} test statistic for the \citet{diebold2002comparing} test on score differentials. All claims of statistical significance in this paper use this criterion.

\begin{table}[t!]
    \renewcommand{\arraystretch}{1.2} 
    \centering  \scriptsize
    \begin{tabular}{|c|r|r|r|r|r|}
    \hline
        scoring rule & LowInfoLucy & ModInfoMuli & HighInfoHannah & PessimisticPenny & OptimisticOmar \\ \hline
        $\CRPS$ & 1.374 & 0.949 & 0.495 & 0.576 & 1.001 \\ 
        $\LogS$ & 2.275 & 1.858 & 0.992 & 1.290 & 1.429 \\ 
        $\LinS$ \dag & -0.122 & -0.186 & -0.502 & -0.670 & -0.251 \\ \hline
    \end{tabular}
    \caption{Mean scores for the five different forecasters from the synthetic experiment, using the scoring rules $\CRPS$, $\LogS$ and $\LinS$. A lower score is intended to convey better predictive performance on average. A dagger symbol \dag \phantom{,}indicates that the scoring rule is not proper. Given a scoring rule, each pairwise performance ranking based on mean scores is statistically significant.}
    \label{tab:scoring rule results}
\end{table}

\subsection{Consistent scoring functions}

While a predictive distribution $F$ for a time-to-event variable $T$ provides the most complete representation of uncertainty, in many practical settings a summary --- such as the mean, a specific percentile, or the interquartile interval (IQI) --- offers greater utility for decision making, communication, or reporting. For example, when informing an impacted community about when major flooding is forecast to commence, the phrase ``likely 6 to 9 hours from now'' is far more accessible than a graph of a CDF. Mappings that present some summary feature of a distribution are called \textit{functionals}. Table~\ref{tab:functional examples} provides examples of functionals for the predictive distributions of Figure~\ref{fig:synthetic example}. In the context of time-to-event forecasts, a functional $U$ is a map from some class $\mathcal{F}$ of probability distributions on the interval $I$ to a subset of $I ^k$ for some $k\geq1$. In many practical contexts, typically $k=1$ and $U(F)$ is a singleton (i.e. $U(F)$ is a single value) from $I$, such as when $U$ is the mean (i.e. expectation) functional. But sometimes $U(F)$ may have multiple values, such as when $U$ is the median functional and the graph of the predictive CDF $F$ is horizontal at height $0.5$. At other times, $U(F)$ may be a vector, such as the interquartile interval ($k=2$).

\begin{table}[t!]
    \renewcommand{\arraystretch}{1.2} 
    \centering  \scriptsize
    \begin{tabular}{|c|r|r|r|r|r|}
    \hline
        functional & LowInfoLucy & ModInfoMuli & HighInfoHannah & PessimisticPenny & OptimisticOmar \\ \hline
        mean & 6.00 & 5.45 & 4.79 & 4.29 & 6.79 \\ 
        median & 5.67 & 5.12 & 4.48 & 4.14 & 5.87 \\ 
        0.1-quantile & 3.15 & 3.55 & 3.90 & 3.84 & 4.11 \\
        IQI & (4.22, 7.42) & (4.18, 6.37) & (4.08, 5.18) & (3.93, 4.48) & (4.65, 7.95) \\ \hline
    \end{tabular}
    \caption{Functional values from the predictive distributions of Figure~\ref{fig:synthetic example}. The 0.1-quantile is equivalent to the 10th percentile and IQI is the interquartile interval.}
    \label{tab:functional examples}
\end{table}

The predictive performance of these distributional summaries can be evaluated using a scoring function $S$, which assigns a score to any given forecast--realization pair. Again we adopt the convention that $S$ is negatively oriented, so that the smaller the score, the better. A desirable property of a scoring function $S$ for time-to-event forecasts is that it is \textit{consistent} for the functional of interest, as per the following definition inspired by \cite{gneiting2011making}. Below, $\mathcal{P}(I^k)$ denotes the power set of $I^k$, which is the collection of all subsets of $I^k$.

\begin{defn}\label{def:consistent}
Suppose that $k\geq 1$. The scoring function $S: I^k \times I \to \RRR$ is \textit{consistent} for the functional $U: \mathcal{F}\to\mathcal{P}(I^k)$ relative to the class $\mathcal{F}$ if
\begin{equation}\label{eq:consistent}
\EEE_F S(u, T) \leq \EEE_F S(x, T)
\end{equation}
for all probability distributions $F$ in $\mathcal{F}$, all $u$ in $U(F)$, and all $x$ in $I^k$. It is \textit{strictly consistent} if it is consistent and equality in (\ref{eq:consistent}) implies that $x\in U(F)$.
\end{defn}

Analogous to proper scoring rules, a scoring function that is consistent for a functional $U$ incentivizes forecasters to report a value from $U(F)$, where $F$ represents their best probabilistic assessment of the time-to-event distribution $T$ given available information. If $T$ is generated from a process with distribution $F$, then values from $U(F)$ yield the optimal expected score.

We summarize some results from the theory of consistent scoring functions \citep{gneiting2011making} and illustrate them using forecasts from the synthetic experiment. A single-valued time-to-event forecast from $I$ is denoted by $x$, with corresponding realization $t$. The squared error scoring function $S(x,t) = (x - t)^2$ is strictly consistent for the mean functional relative to the class of probability distributions on $I$ with finite second moment. The first row of Table~\ref{tab:scoring function results} shows the mean squared error (MSE) for each forecaster when using the mean (expected value) of their predictive distribution. Rankings based on MSE agree with (\ref{eq:rankings}). The absolute error scoring function $S(x,t) = |x - t|$ is consistent for the median functional relative to distributions with finite first moment (e.g., \citealp{raiffa1961applied}), but not for the mean functional. The second row of Table~\ref{tab:scoring function results} illustrates that using a scoring function inconsistent with the functional can lead to misleading inferences; here, PessimisticPenny outperforms HighInfoHannah, contrary to (\ref{eq:rankings}). The third row shows that using a strictly consistent scoring function --- absolute error for the median functional --- restores rankings consistent with (\ref{eq:rankings}).

\begin{table}[t!]
    \renewcommand{\arraystretch}{1.2} 
    \centering  \scriptsize
    \begin{tabular}{|c|c|r|r|r|r|r|}
    \hline
        scoring function & functional & LowInfoLucy & ModInfoMuli & HighInfoHannah & PessimisticPenny & OptimisticOmar \\ \hline
        squared error & mean & 6.189 & 3.066 & 0.987 & 1.229 & 5.021 \\ 
        absolute error \dag & mean & 1.954 & 1.359 & 0.729 & 0.706 & 2.106 \\ 
        absolute error & median & 1.939 & 1.335 & 0.686 & 0.754 & 1.330 \\ 
        $\QL_{0.9}$ & 0.9-quantile & 0.506 & 0.372 & 0.229 & 0.325 & 0.593 \\ 
        $\IS_{0.5}$ & IQI & 1.545 & 1.068 & 0.557 & 0.640 & 1.125 \\ \hline
    \end{tabular}
    \caption{Mean scores for each forecaster from the synthetic experiment, with forecasts generated from a functional of their predictive distribution. In each case, a lower score is desirable. A dagger symbol \dag \phantom{,}indicates that the scoring function is not consistent for the paired functional. Given a scoring method, each pairwise performance ranking based on mean scores is statistically significant apart from absolute error on the median forecasts for the comparison ModInfoMuli--OptimisticOmar.}
    \label{tab:scoring function results}
\end{table}

If $0<\alpha<1$ then the quantile loss scoring function $\QL_\alpha$, given by
\begin{equation}\label{eq:check loss}
\QL_\alpha(x,t) = \begin{cases}
\alpha(t - x) & \text{if } x < t \\
(1-\alpha)(x - t) & \text{if } x \geq t,
\end{cases}
\end{equation}
is strictly consistent for the $\alpha$-quantile functional relative to the class of probability distributions with finite first moment \citep{raiffa1961applied}. By taking $\alpha=0.5$, one attains the result that the absolute error is strictly consistent for the median. The fourth row of Table~\ref{tab:scoring function results} gives the mean quantile loss for 0.9-quantile (90th percentile) time-to-event forecasts, with rankings that agree with (\ref{eq:rankings}). Quantile loss is not the only strictly consistent scoring function for the quantile functional. If $g:I\to\RRR$ is a strictly increasing function then the scoring function $S:I\times I\to\RRR$, given by
\begin{align}
S(x,t) 
&= (\one\{x\geq t\} - \alpha)(g(x)-g(t)) \label{eq:consistent quantile sf} \\
&= \begin{cases}
\alpha(g(t) - g(x)) & \text{if } x < t \\
(1-\alpha)(g(x) - g(t)) & \text{if } x \geq t,
\end{cases} \notag
\end{align}
is strictly consistent relative to the class of probability distributions $F$ on $I$ for which $\EEE_F g(T)$ exists and is finite \citep{thomson1978eliciting, saerens2000building, gneiting2011quantiles}. This fact will be critical for dealing with right-censored data in the next section.

The IQI is a special case of the central $(1-\alpha)\times 100$\% prediction interval $(x_1, x_2)$, where $x_1$ is an $\tfrac{\alpha}{2}$-quantile and $x_2$ is a $(1-\tfrac{\alpha}{2})$-quantile of the predictive distribution $F$. The interval scoring function $\IS_{\alpha}: I^2 \times I \to [0,\infty)$ is given by
\begin{align}
\IS_\alpha((x_1, x_2), t)
&=\QL_{\alpha/2}(x_1,t) + \QL_{1-\alpha/2}(x_2,t) \\
&=\one\{t<x_1\}(x_1-t) + \tfrac{\alpha}{2}(x_2-x_1) + \one\{t>x_2\}(t-x_2) \label{eq:IS clear form}
\end{align}
for all $(x_1, x_2)$ in $I^2$ and $t$ in $I$ \citep{gneiting2007strictly}. Equation~(\ref{eq:IS clear form}) decomposes $\IS_{\alpha}$ into three components: overprediction penalty, interval-width penalty, and underprediction penalty. Narrow intervals that contain the realization are rewarded. When $x_1=x_2$, the score reduces to the absolute error. Because $\IS_{\alpha}$ is the sum of two strictly consistent scoring functions for quantiles, it is strictly consistent for the central prediction interval functional relative to distributions with finite first moment \citep[Lemma~2.6]{fissler2016higher}. The last row of Table~\ref{tab:scoring function results} shows mean interval scores for each forecaster’s IQI predictions in the synthetic experiment, generating rankings consistent with (\ref{eq:rankings}). More generally, $\QL_{\alpha_1}+\QL_{\alpha_2}$ is strictly consistent for intervals whose endpoints are $\alpha_1$- and $\alpha_2$-quantiles, relative to the class of distributions with finite first moment.

\begin{defn}\label{def:elicitability}\citep{lambert2008eliciting}
A functional $U$ is \textit{elicitable} relative to the class of probability distributions $\mathcal{F}$ if there exists a scoring function $S$ that is strictly consistent relative to $\mathcal{F}$.
\end{defn}

It is important that a functional is elicitable, else it remains unclear how one might rank competing forecasters while incentivizing truthful reporting of the functional value of their predictive distributions. Many commonly used functionals are elicitable under mild regularity conditions, including the mean, median, quantiles and prediction intervals whose endpoints are quantiles \citep{savage1971elicitation, thomson1978eliciting, saerens2000building}, while the mode is not elicitable \citep{heinrich2014mode}.
In the next section we investigate which of these functionals remain elicitable when evaluation is performed using right-censored data.

\section{Evaluation with right-censored observations and forecasts}\label{s:right censored evaluation}

We now consider evaluation relative to a common fixed censoring time $\tau$. While the realized event time $t$ may be unknown, we assume that the right-censored realization $\min(t,\tau)$ is available and can be used for evaluation. Throughout, assume that the right-censoring time $\tau$ is positive and belongs to the interior of the time interval $I$. To simplify notation, given a real number $t$, a vector $x=(x_1,\ldots,x_k)$ in $\RRR^k$, a random variable $T$ taking values in $I$, a subset $A$ of real numbers and a CDF $F:I\to[0,1]$, let $[\,\cdot\,]_\tau$ denote their right-censored counterparts as follows:
\begin{align*}
[t]_\tau &= \min(t,\tau) \\
[x]_\tau &= \big(\min(x_1,\tau),\ldots,\min(x_k,\tau)\big) \\
[T]_\tau &=  \min(T,\tau) \\
[A]_\tau &= \{\min(a,\tau): a\in A \} \\
[F]_\tau(s) &=
\begin{cases}
F(s), & s < \tau, \\
1, & s \geq \tau.
\end{cases}
\end{align*}

\subsection{Strictly locally proper scoring rules}\label{ss:provisional scoring rules}

To illustrate the need for a rigorous approach to evaluation under fixed-horizon right-censoring, consider assessing the CDF forecasts from the synthetic experiment after 4 months. Suppose we use the strictly proper scoring rule $\CRPS$ but restrict evaluation to cases where the event has occurred (i.e., $t_i \leq 4$), which accounts for about 21\% of the 10,000 forecast cases. The first row of Table~\ref{tab:provisional results} shows the mean $\CRPS$ for each forecaster, and it is immediately clear that (\ref{eq:rankings}) is violated: HighInfoHannah does not achieve the lowest mean score. This occurs because the optimal strategy under this evaluation method is to forecast the $\tau$-truncated version of one’s true probabilistic beliefs rather than the full distribution.

\begin{table}[t!]
    \renewcommand{\arraystretch}{1.2} 
    \centering  \scriptsize
    \begin{tabular}{|c|c|r|r|r|r|r|}
    \hline
        forecast & score method & LowInfoLucy & ModInfoMuli & HighInfoHannah & PessimisticPenny & OptimisticOmar \\ \hline
        CDF & $\CRPS$ if $t\leq4$ \dag & 1.749 & 0.841 & 0.315 & 0.237 & 1.127 \\ 
        CDF & $\survCRPS_2$ \dag & 0.059 & 0.025 & 0.007 & 0.005 & 0.025 \\
        CDF & $\twCRPS_6$ & 0.627 & 0.440 & 0.232 & 0.275 & 0.380 \\ 
        CDF & $\twCRPS_{12}$ & 1.339 & 0.920 & 0.479 & 0.558 & 0.943 \\
        CDF & $\twLogS_6$ & 1.475 & 1.140 & 0.546 & 0.711 & 0.786 \\ 
        CDF & $\twLogS_{12}$ & 2.243 & 1.827 & 0.969 & 1.260 & 1.397 \\
        0.9-quantile & $\twQL_{0.9,6}$ & 0.096 & 0.096 & 0.082 & 0.126 & 0.096 \\ 
        0.9-quantile & $\twQL_{0.9,12}$ & 0.474 & 0.350 & 0.217 & 0.311 & 0.509 \\ 
        IQI & $\twIS_{0.5,6}$ & 0.694 & 0.494 & 0.262 & 0.309 & 0.413 \\ 
        IQI & $\twIS_{0.5,12}$ & 1.510 & 1.037 & 0.539 & 0.621 & 1.075 \\ \hline
    \end{tabular}
    \caption{Mean scores for each of the forecasters from the synthetic experiment, for different forecast types and different methods for dealing with right-censored realizations. The final subscript in each scoring rule or function indicates the right-censoring time in months. In each case, a lower score is desirable. A dagger symbol \dag \phantom{,}indicates that the scoring method is not proper. Given a scoring method, each pairwise performance ranking based on mean scores is statistically significant apart from $\twQL_{0.9,6}$ for the comparisons LowInfoLucy--ModInfoMuli and ModInfoMuli--OptimisticOmar.}
     \label{tab:provisional results}
\end{table}

To develop sound methods for evaluation under fixed-horizon right-censoring, we require an analogue of strictly proper scoring rules, adapted so that a predictive distribution $G$ is scored against the right-censored realization. In this setting, $\EEE_F S(G,[T]_\tau)$ denotes the expectation of the score $S(G,[T]_\tau)$, where the time-to-event random variable $T$ has distribution $F$.

\begin{defn}\label{def:provisionally strictly proper}
A scoring rule $S:\mathcal{F}\times I\to(-\infty,\infty]$ is \textit{locally proper} on the interval $[0,\tau)$ relative to the class of probability distributions $\mathcal{F}$ if
\begin{equation}\label{eq:provisionally proper}
\EEE_F S(F, [T]_\tau) \leq \EEE_F S(G, [T]_\tau)
\end{equation}
for all $F$ and $G$ in $\mathcal{F}$. It is  \textit{strictly locally proper} on $[0,\tau)$ if it is locally proper and equality in (\ref{eq:provisionally proper}) implies that $[G]_\tau=[F]_\tau$.
\end{defn}

Thus a strictly locally proper scoring rule $S$ on $[0,\tau)$ encourages forecasters to quote any distribution that aligns with their best judgment on the interval $[0,\tau)$, given evaluation using the right-censored realization.

Definition~\ref{def:provisionally strictly proper} is a specialization of the notion of strict local propriety of \citet{holzmann2017focusing} and \citet[Definition~4]{de2025localizing} to the setting where localization is induced by right-censoring at a fixed horizon $\tau$. It is also closely related to ideas of \citet{rindt2022survival}, who essentially prove that several evaluation methods used in the survival analysis literature are not proper when using right-censored observations. One such scoring rule is the survival-CRPS \citep{avati2020countdown} for right-censored data, which for right-censoring time $\tau$ is given by
\begin{equation}\label{eq:survival CRPS}
\survCRPS_\tau(F, t)=
\begin{cases}
\CRPS(F,t) & \text{if } t<\tau \\
\displaystyle\int_0^\tau F(s)^2\,\dd s & \text{if } t\geq \tau.
\end{cases}
\end{equation}
Mean survival-CRPS results when $\tau=2$ for the synthetic experiment are presented in the second row of Table~\ref{tab:provisional results}, with implied rankings violating (\ref{eq:rankings}).

The following theorem allows one to construct strictly locally proper scoring rules from strictly proper scores. It is a special case of Theorem 2 of \citet{de2025localizing}, applied to fixed-horizon right-censoring. We state it separately because of its simple form and practical relevance for time-to-event forecast evaluation.

\begin{thm}\label{thm:provisional from strictly proper 1}
Suppose that $\mathcal{F}$ and $\mathcal{F}_*$ are two classes of distributions on $I$ such that $[F]_\tau\in\mathcal{F}_*$ whenever $F\in\mathcal{F}$. If 
$S$ is a strictly proper scoring rule relative to $\mathcal{F}_*$ then the scoring rule $S_\tau$, defined by $S_\tau(F,t)=S([F]_\tau,t)$, is strictly locally proper on $[0,\tau)$ relative to $\mathcal{F}$.
\end{thm}

We apply this theorem to the $\CRPS$ and $\LogS$, with threshold-weighted versions of these emerging as strictly locally proper scoring rules. Given a weight function $w:I\to[0,1]$, the threshold-weighted continuous ranked probability score ($\twCRPS$) is defined by
\begin{equation}\label{eq:twcrps}
\twCRPS(F, t; w) =  \int_0^\infty (\one\{s\geq t\} - F(s))^2 \, w(s) \,\dd s ,
\end{equation}
for all $t$ in $I$ and all predictive distributions $F$ for which the integral exists and is finite. The $\twCRPS$ can also be expressed as
\begin{align}
\twCRPS(F, t; w) &= \EEE_F|v(X) - v(t)| - \frac{1}{2} \EEE_F|v(X) - v(X')|, \label{eq:twcrps chaining}
\intertext{where}
v(x) &= \int_{[0,x)}w(s)\,\dd s \label{eq:chaining function}
\end{align}
for all $x$ in $I$ and where $X$ and $X'$ are independent random variables with distribution $F$. The function $v$ of (\ref{eq:chaining function}), called a chaining function \citep{allen2023evaluating}, once calculated facilitates fast computation of $\twCRPS$ via (\ref{eq:twcrps chaining}) when the CDF $F$ is derived from an ensemble of forecasts.

The threshold-weighted logarithmic score $\twLogS$ (also known as the censored likelihood score) of \citet{diks2011likelihood} is given by
\begin{equation}\label{eq:twLogS}
\twLogS(F, t; w) = -w(t)\ln F'(t) - (1-w(t))\ln\Big( 1 - \int_I w(s)F'(s)\,\dd s \Big),
\end{equation}
where $F'$ denotes the density of $F$.

For the special case when $w(s)=\one\{0\leq s < \tau\}$, one obtains
\begin{align}
\twCRPS_\tau(F, t)
&= \int_0^\tau (\one\{s\geq t\} - F(s))^2 \, \dd s \label{eq:twcrps specific} \\
& = \EEE_F\big|[X]_\tau - [t]_\tau\big| - \tfrac{1}{2} \EEE_F\big|[X]_\tau - [X']_\tau\big|. \label{eq:twcrps chaining specific}
\end{align}
and
\begin{equation}\label{eq:twlogs special case}
\twLogS_\tau(F, t) = -\one\{0\leq t < \tau\}\ln F'(t) - \one\{t \geq \tau\} \ln(1 - F(\tau)).
\end{equation}
The latter yields a score that combines density evaluation below $\tau$ with a binary logarithmic score for survival beyond $\tau$.

The following corollary of Theorem~\ref{thm:provisional from strictly proper 1} is a special case of the localization framework of \citet{de2025localizing}, stated here in a form tailored to the evaluation of time-to-event forecasts.

\begin{cor}\label{cor:twCRPS}
Suppose that the weight function $w:I\to[0,1]$ is strictly positive on $[0,\tau)$ and zero elsewhere.
\begin{enumerate}
\item[(i)] The $\twCRPS$ with weight function $w$ is strictly locally proper on the interval $[0,\tau)$ relative to the class $\mathcal{F}$ of probability distributions on $I$.
\item[(ii)] The $\twLogS$ with weight function $w$ is strictly locally proper on the interval $[0,\tau)$ relative to the class of probability distributions on $I$ with a probability density function.
\end{enumerate}
\end{cor}

Theorem~\ref{thm:provisional from strictly proper 1} can also be applied to a wide range of other strictly proper scoring rules. See \cite{de2025localizing} for applications in the more general setting of strictly locally proper scoring rules.

Mean $\twCRPS_\tau$ and $\twLogS_\tau$ results for the synthetic experiment when $\tau$ is 6 and 12 are presented in the third to sixth rows of Table~\ref{tab:provisional results}, with rankings that are consistent with (\ref{eq:rankings}). Efficient computation of $\twCRPS_\tau$ is enabled by a closed-form expression for $\twCRPS_\tau(F,t)$ when $F$ is a gamma distribution. This is readily derived by subtracting the closed formula for the CRPS of the left-censored $F$ \citep[Eq~10]{scheuerer2015statistical} from the closed formula for the CRPS of $F$ \citep{scheuerer2015probabilistic}.

Of interest is that $\twCRPS_\tau$ ranks ModInfoMuli behind OptimisticOmar when $\tau=6$, but that this pairwise ranking is reversed when $\tau=12$. 
Insight into how $\twCRPS_\tau$ rankings depend on the censoring time $\tau$ can be gained by plotting the integrand of the CRPS in (\ref{eq:CRPS1}) against the time variable $s$, a diagnostic approach dating back to \cite{schumacher2003assess}. The integrand $(\one\{s\geq t\} - F(s))^2$ is the squared difference between the predictive probability $F(s)$ that the time-to-event does not exceed time $s$ and the binary realization $\one\{s\geq t\}$ indicating whether $s \geq t$. This squared difference is the Brier score \citep{brier1950verification}, a strictly proper scoring rule for probabilistic predictions with binary outcomes. Figure~\ref{fig:murphy synthetic}a plots the mean Brier score with event threshold $s$ against $s$, where one may interpret $s$ as a user decision threshold. A lower mean Brier score for threshold $s$ indicates better predictive performance for users with that threshold averaged across all forecast cases. The mean CRPS equals the area under the curve, while the mean $\twCRPS_\tau$ is the area under the curve on the interval $0\leq s \leq \tau$. As expected, HighInfoHannah performs best across all decision thresholds. Notably, OptimisticOmar outperforms ModInfoMuli for thresholds up to about 6.5 months, while ModInfoMuli dominates OptimisticOmar beyond 7 months. This occurs because OptimisticOmar leverages knowledge of $x_i + y_i$ and excels when $x_i + y_i \leq s$, whereas for larger $s$, ModInfoMuli's ideal forecasts based on $x_i$ outperform OptimisticOmar's systematic overpredictions. If only these two forecasters were available, OptimisticOmar's forecasts would be preferable for users with decision thresholds within the first 6 to 7 months, and ModInfoMuli's preferable beyond that horizon.

\begin{figure}[t!]
\centering
\includegraphics[width=0.75\textwidth]{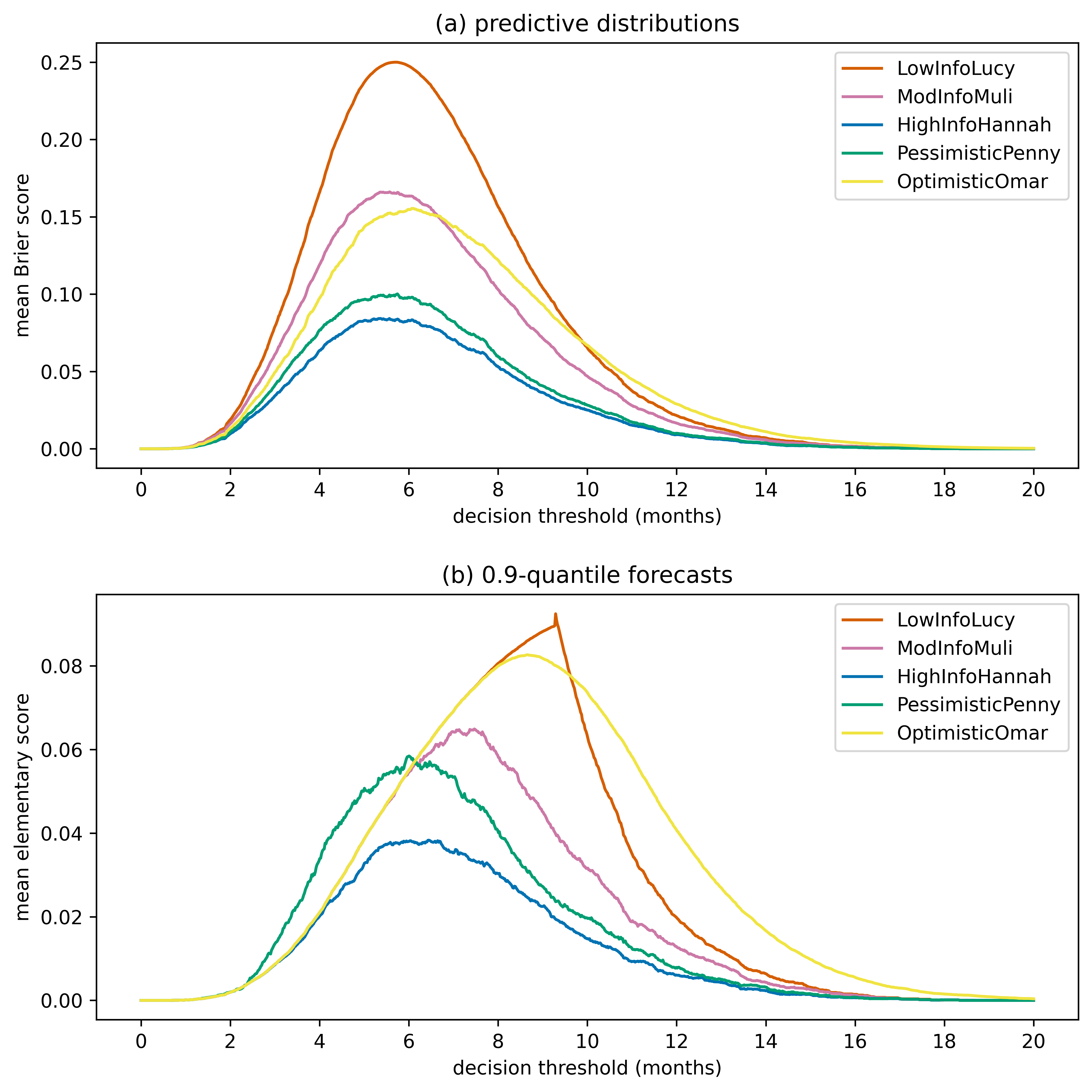}
\caption{(a) Brier score decomposition of CRPS for predictive distributions and (b) elementary score decomposition of quantile loss for 0.9-quantile forecasts, against decision threshold, for the forecasters from the synthetic experiment. For each decision threshold, a lower mean score is preferred.}
\label{fig:murphy synthetic}
\end{figure}

\subsection{Locally elicitable functionals}\label{ss:provisional scoring functions}

While the theory underlying the evaluation of time-to-event predictive distributions under fixed-horizon right-censoring, presented in Section~\ref{ss:provisional scoring rules}, is an application of existing scoring-rule theory, corresponding questions for functionals such as quantiles, prediction intervals and expectations under fixed-horizon right-censoring appear not to have been previously studied. The main theoretical results in this section, concerning local consistency and local elicitability under fixed-horizon right-censoring, therefore appear to be new.

\begin{defn}\label{def:provisionally strictly consistent}
Suppose that $k\geq1$ and that $U$ is a functional whose values are subsets of $I^k$. A scoring function $S$ is \textit{locally consistent} for the functional $U$ relative to the class $\mathcal{F}$ on the interval $[0,\tau)$ if
\begin{equation}\label{eq:censored consistency}
\EEE_F S(u,[T]_\tau) \leq \EEE_F S(x,[T]_\tau)
\end{equation}
for all probability distributions $F$ in $\mathcal{F}$, all $u$ in $U(F)$ and all $x$ in $I^k$. It is \textit{strictly locally consistent} for $U$ relative to $\mathcal{F}$ on $[0,\tau)$ if it is locally consistent and equality in (\ref{eq:censored consistency}) implies that $[x]_\tau\in [U(F)]_\tau$.
\end{defn}

Note that for strict local consistency, we do not insist that the only minimizers of the expected score are functional values; rather, that the minimizers, when right-censored, agree with right-censored functional values. Observing that equality of right-censored sets defines an equivalence relation, we have the following: if a scoring function is strictly locally consistent, then the only minimizers, `modulo right-censorship', of a forecaster's expected score are functional values from the predictive distribution that aligns with their true beliefs.

\begin{defn}\label{def:provisionally elicitable}
The functional $U$ is \textit{locally elicitable} relative to the class $\mathcal{F}$ on the interval $[0,\tau)$ if there exists a scoring function $S$ that is strictly locally consistent for $U$ relative to the class $\mathcal{F}$ on the interval $[0,\tau)$. It is \textit{completely locally elicitable} relative to the class $\mathcal{F}$ if it is both elicitable relative to the class $\mathcal{F}$ and locally elicitable relative to the class $\mathcal{F}$ on $[0,\tau)$ for every $\tau$ in the interior of $I$.
\end{defn}

The notion of complete local elicitability is important for the evaluation of long- or infinite-horizon time-to-event forecasts of rare events. It identifies functionals that admit coherent evaluation both under fixed-horizon right-censoring and when complete event-time information is available.

We begin with a positive result for quantile functionals, which draws on the threshold-weighted quantile scoring functions proposed by \citet{taggart2022evaluation}.

\begin{thm}\label{thm:quantile is provisionally elicitable}
Suppose that $0<\alpha<1$ and that $\tau$ belongs to the interior of $I$. Let $\mathcal{F}$ denote the class of probability distributions on the interval $I$.
\begin{enumerate}
\item[(a)] If the function $g:I\to\RRR$ is strictly increasing on the interval $[0,\tau]$, then the scoring function $S$, given by
\begin{equation}\label{eq:consistent provisional quantile sf}
S(x,t) = (\one\{[x]_\tau\geq [t]_\tau\} - \alpha)(g([x]_\tau)-g([t]_\tau)),
\end{equation}
is strictly locally consistent for the $\alpha$-quantile functional relative to $\mathcal{F}$ on $[0,\tau)$.
\item[(b)] The $\alpha$-quantile functional is completely locally elicitable with respect to $\mathcal{F}$.
\end{enumerate}
\end{thm}

Theorem~\ref{thm:quantile is provisionally elicitable}(a) gives methodological justification for evaluating an $\alpha$-quantile forecast $x$ against a right-censored observation $[t]_\tau$ to generate the score $S(x,[t]_\tau)$, where $S$ is defined by (\ref{eq:consistent provisional quantile sf}). The proof of the theorem is in Appendix~\ref{a:proofs} and has two core ideas. First, if $g$ is strictly increasing on $[0,\tau]$ then so too is the function $s\mapsto g([s]_\tau)$, and existing results about the scoring function (\ref{eq:consistent quantile sf}) can be applied. Second, when $U$ is a quantile functional, right-censoring and functional evaluation of distributions are interchangeable in the sense that $U([F]_\tau)=[U(F)]_\tau$.

By taking $g(s)=s$ as a special case of (\ref{eq:consistent provisional quantile sf}), one obtains the threshold-weighted quantile loss $\twQL_{\alpha,\tau}$, given by
\begin{equation}\label{eq:twCL}
\twQL_{\alpha,\tau}(x,t) = (\one\{[x]_\tau\geq [t]_\tau\} - \alpha)([x]_\tau-[t]_\tau).
\end{equation}
Mean scores $\overline{\twQL}_{0.9,\tau}$ of 0.9-quantile forecasts from the synthetic experiment are shown in Table~\ref{tab:provisional results} when evaluated against right-censored realizations for right-censoring time $\tau$. Ranks based on these mean scores agree with (\ref{eq:rankings}), with some qualification regarding the pairwise ranking of LowInfoLucy and ModInfoMuli when $\tau=6$. LowInfoLucy's mean score is slightly greater than ModInfoMuli's, but the difference is negligible ($1.04\times10^{-4}$) and the result is not statistically significant over 10,000 forecast cases. This occurs because $\twQL_{\alpha,\tau}(x,t)$ depends only on the right-censored forecast $[x]_\tau$ and realization $[t]_\tau$, and when $\tau=6$ LowInfoLucy's and ModInfoMuli's right-censored 0.9-quantile forecasts take the same value of 6 in about 97\% of forecast cases. Consequently, their score differential is zero for most cases, making significance difficult to achieve with a sample size of 10,000. Statistical significance is attained when the experiment is repeated over 500,000 cases, with ModInfoMuli ranked better. This issue does not arise for 0.1-quantile forecasts, which frequently differ even when $\tau=6$. The Murphy diagram, discussed next, provides further insight.

As with $\twCRPS$, relative rankings of misspecified forecast procedures under strictly locally consistent scoring functions can vary with the censoring time $\tau$. The Murphy diagram, which is a plot of mean elementary score against decision threshold $\theta$, is a useful diagnostic for understanding this. A Murphy diagram for 0.9-quantile forecasts from the synthetic experiment is presented in Figure~\ref{fig:murphy synthetic}b.
The theory behind Murphy diagrams for quantile forecasts, and an interpretation using simple optimal user decision models, was developed by \citet{ehm2016quantiles}. We summarize a few key points below.
\begin{itemize}
\item The elementary score $\ES_{\alpha,\theta}$ for an $\alpha$-quantile and decision threshold $\theta$ in $I$ is given by
\[
\ES_{\alpha,\theta}(x,t)=
\begin{cases}
1-\alpha & \text{if } t\leq\theta<x \\
\alpha & \text{if } x\leq\theta<t \\
0 & \text{otherwise,}
\end{cases}
\]
where $x$ is the real-valued forecast and $t$ is the time-to-event realization. Thus $\ES_{\alpha,\theta}$ only assigns a positive penalty for cases where the forecast and realization fall on opposite sides of $\theta$. Conveniently, $\ES_{\alpha,\theta}(x,t)=\ES_{\alpha,\theta}(x,[t]_\tau)$ whenever $\theta \leq \tau$, facilitating the calculation of elementary scores using right-censored realizations.
\item Suppose that, following the classical simple cost--loss user decision model \citep{richardson2012economic}, a forecast user follows an optimal decision rule to act if the $\alpha$-quantile forecast $x$ exceeds a particular decision threshold $\theta$, and to take no action otherwise. (E.g., a maintenance engineer may decide to replace a component if and only if the median time-to-failure forecast for the component is no greater than 4 years.) If forecaster A has a lower mean elementary score $\overline{\ES}_{\alpha,\theta}$ than forecaster B for that decision threshold, then over that set of forecast cases the user would have made better decisions using the quantile forecasts of forecaster A than those of forecaster B.
\item The area under a forecaster's Murphy diagram curve equals their mean quantile loss:
\[\overline{\QL}_\alpha = \int_I \overline{\ES}_{\alpha,\theta}\,\dd\theta.\]
\item The area under a forecaster's Murphy diagram curve from decision threshold $0$ to $\tau$ equals their mean threshold-weighted quantile loss:
\[\overline{\twQL}_{\alpha,\tau} = \int_0^\tau \overline{\ES}_{\alpha,\theta}\,\dd\theta.\]
\item The previous two points are special cases of the following. Suppose that $S$ is the strictly locally consistent scoring function of Equation~(\ref{eq:consistent provisional quantile sf}) with corresponding increasing function $g$. If $g$ is differentiable with derivative $g'$, then the mean score $\overline{S}$ equals the threshold-weighted area under the Murphy curve from decision threshold 0 to $\tau$, where the weight for decision threshold $\theta$ is $g'(\theta)$:
\[\overline{S} = \int_0^\tau \overline{\ES}_{\alpha,\theta}\,g'(\theta)\,\dd\theta.\]
\end{itemize}
Using Figure~\ref{fig:murphy synthetic}b, OptimisticOmar outperforms PessimisticPenny for users with decision thresholds up to about 6 months, while PessimisticPenny performs better beyond 7 months. When averaging, with equal weight, performance across all decision thresholds up to 6 months, OptimisticOmar performs better on average than PessimisticPenny (Table~\ref{tab:provisional results} row 7), whereas when averaging across all decision thresholds up to 12 months PessimisticPenny performs better (Table~\ref{tab:provisional results} row 8). The Murphy curves for LowInfoLucy and ModInfoMuli are nearly indistinguishable for thresholds up to 6 months, consistent with their similar mean for $\twQL_{0.9,6}$. Only for thresholds exceeding 6 months does the performance gap emerge, as their right-censored 0.9-quantile forecasts begin to differ more frequently.

The discussion on Figure~\ref{fig:murphy synthetic} gives a decision-theoretic interpretation of local evaluation on the time interval $[0,\tau)$ when using threshold-weighted CRPS and quantile loss. The mean threshold-weighted scores are an aggregate measure of forecast accuracy for forecast users who have time-to-event decision thresholds ranging from $0$ to $\tau$. They do not measure accuracy for users whose decision thresholds exceed $\tau$.

Since a vector-valued functional constructed from real-valued locally elicitable functionals is locally elicitable (see Lemma~\ref{lem:vector valued functional} in Appendix~\ref{a:proofs}), the central prediction interval functional is also completely locally elicitable.

\begin{thm}\label{thm:prediction interval is provisionally elicitable}
Suppose that $0<\alpha<1$ and that $\tau$ belongs to the interior of $I$. Let $\mathcal{F}$ denote the class of probability distributions on the interval $I$.
\begin{enumerate}
\item[(a)] The threshold-weighted interval scoring function $\twIS_{\alpha,\tau}:I^2\times I\to\RRR$, given by
\begin{equation}\label{eq:twIS2}
\twIS_{\alpha,\tau}((x_1,x_2), t)
= \frac{\alpha}{2}([x_2]_\tau-[x_1]_\tau) + \one\{[x_1]_\tau>[t]_\tau\}\big([x_1]_\tau-[t]_\tau)+ \one\{[x_2]_\tau<[t]_\tau\}([t]_\tau-[x_2]_\tau)\,,
\end{equation}
is strictly locally consistent for the $(1-\alpha)\times100\%$ central prediction interval relative to $\mathcal{F}$ on $[0,\tau)$.
\item[(b)] The $(1-\alpha)\times100\%$ central prediction interval is completely locally elicitable with respect to $\mathcal{F}$.
\end{enumerate}
\end{thm}

The scoring function $\twIS_{\alpha,\tau}$, which appears to be novel, equals the sum of $\twQL_{\alpha_1,\tau}$ and $\twQL_{\alpha_2,\tau}$. The representation (\ref{eq:twIS2}) shows that the threshold-weighted interval score decomposes into the distance from the right-censored observation to the right-censored prediction interval plus a penalty proportional to the  width of the right-censored prediction interval. When $\alpha=0.5$ the central prediction interval becomes the IQI. Mean scores $\overline{\twIS}_{0.5,\tau}$ for IQI predictions from the synthetic experiment appear in the final two rows of Table~\ref{tab:provisional results}, yielding rankings consistent with (\ref{eq:rankings}).
Theorem~\ref{thm:prediction interval is provisionally elicitable} is a special case of a more general result that prediction intervals of the form $(x_1,x_2)$ where $x_1$ is an $\alpha_1$-quantile and $x_2$ is an $\alpha_2$-quantile ($0<\alpha_1<\alpha_2<1$) are completely locally elicitable, via scoring functions that are appropriate linear combinations of (\ref{eq:consistent provisional quantile sf}).

The next theorem gives general conditions under which local elicitability on $[0,\tau)$ fails. We use it to establish a negative result: the mean (i.e., expected value) functional is not locally elicitable.

\begin{thm}\label{thm:not prov elic}
Suppose that $\tau$ belongs to the interior of $I$, $\mathcal{F}$ is a class of probability distributions on $I$ and $U:\mathcal{F}\to\mathcal{P}(I)$ is a functional on $\mathcal{F}$. If $\mathcal{F}$ contains two distributions $F_1$ and $F_2$ such that
\begin{enumerate}
\item[(a)] $U(F_1) \subset(0,\tau)$ and $U(F_2) \subset(0,\tau)$ 
\item[(b)] $U(F_1) \neq U(F_2)$ and
\item[(c)] $[F_1]_\tau = [F_2]_\tau$,
\end{enumerate}
then $U$ is not locally elicitable with respect to $\mathcal{F}$ on $[0,\tau)$.
\end{thm}

\begin{cor}\label{cor:expectation not provisionally elicitable}
Suppose that $\tau$ is in the interior of $I$ and that $m$ is a positive integer. The mean functional $U$ is not locally elicitable on $[0,\tau)$ relative to any class $\mathcal{F}$ of probability distributions on $I$ that contains either (i) the distributions on $I$ with finite support, or (ii) the distributions on $I$ with a probability density function and finite $m$th moment.
\end{cor}

The key reason for this negative result is that, unlike the quantile functional, right-censoring and functional evaluation are not interchangeable for the mean functional. For the same reason, Theorem~\ref{thm:not prov elic} can be used to show that expectiles, Huber functionals and the midpoint of modal intervals are not locally elicitable, even though they are elicitable \citep{gneiting2011making, taggart2022point, ferguson1967mathematical, lehmann1998theory}. Extending the elicitation theory presented here to more general localization schemes, such as those considered by \citet{de2025localizing}, appears to be a natural direction for future work.

\subsection{Dealing with right-censored forecasts}\label{ss:right-censored forecasts}

Sections~\ref{ss:provisional scoring rules} and \ref{ss:provisional scoring functions} addressed forecast evaluation using right-censored realizations. We now consider the case where forecasts themselves are right-censored. This situation arises naturally whenever a forecasting system has a finite prediction horizon $\tau$. In such cases, the exact forecast value $x$ may be unknown, but its right-censored counterpart $[x]_\tau$ remains known.

Fortunately, the strictly locally proper scoring rules and strictly locally consistent scoring functions of Sections~\ref{ss:provisional scoring rules} and \ref{ss:provisional scoring functions} are invariant with respect to right-censorship of the forecast. That is, the scoring function $S_\tau$ of Theorem~\ref{thm:provisional from strictly proper 1}, the threshold-weighted CRPS (\ref{eq:twcrps}, \ref{eq:twcrps specific}) and threshold-weighted logarithmic score (\ref{eq:twLogS}, \ref{eq:twlogs special case}) can be calculated if the values $F(s)$ of the predictive distribution are only known for $0 \leq s < \tau$, so that $[F]_\tau$ is known even if $F$ is unknown. Similarly, the values of the strictly locally consistent scores (\ref{eq:consistent provisional quantile sf}, \ref{eq:twCL}, \ref{eq:twIS2}) for quantile and interval forecasts are unchanged if the forecast $x$ is replaced with $[x]_\tau$. The CRPS decomposition and Murphy diagram diagnostic tools illustrated in Figure~\ref{fig:murphy synthetic} are also invariant with respect to right-censorship of the forecast, provided that one restricts the diagrams to decision thresholds from $0$ to $\tau$.

In summary, right-censored CDF, quantile and interval forecasts can be evaluated using the threshold-weighted scores introduced earlier. As illustrated in Section~\ref{s:real examples}, such forecasts can also be derived from ensembles of time-series forecasts with finite prediction horizons. However, right-censored mean (expected value) forecasts cannot be derived from an ensemble if the event fails to occur in any member within the model's prediction horizon.

Right-censorship of forecasts at time $\tau$ can also be used in situations where the predicted probability that the event never occurs is positive, so that the outcome space is the extended half real line $[0,\infty]$ and $\PPP(T=\infty)>0$, where $T=\infty$ denotes that the event never occurs. In such cases, a predictive distribution $F$, defined by $F(s)=\PPP(T\leq s)$ whenever $s\in[0,\infty]$, has right-censored counterparts $[F]_\tau$ that can be used for evaluation as discussed above.

\section{Examples from hydrological and weather prediction}\label{s:real examples}

\subsection{First passage time forecasts for flood}\label{ss:example flood}

Forecasting the first time a river exceeds a critical flood level, relative to model initialization time, is an example of a first passage time forecast. To illustrate evaluation methods, we consider river height predictions for eight locations in the Nepean--Hawkesbury River catchment, Australia. Forecasts are generated by an ensemble prediction system (EPS) comprising $m$ precipitation forecasts converted via a hydrological model into $m$ river height forecasts at each location. We compare two EPSs used by the Australian Bureau of Meteorology, both issuing hourly river height forecasts to a 168-hour (7-day) horizon using the same hydrological model. Issuance times are every 3 hours when there is a flood risk and daily otherwise, with 665 issuance times spanning 9~May 2024 to 31~July 2025. EPS~A has 12 ensemble members with precipitation taken from 12 different numerical weather prediction models. EPS~B has 32 members using stochastic precipitation forecasts for the short term, transitioning to numerical weather prediction thereafter. Figure~\ref{fig:hnv hydrograph}a,b shows ensemble river height forecasts (gray lines) for North Richmond initialized at 6~June~2024, observed heights (black line), and critical flood thresholds (horizontal lines).

\begin{figure}[t!]
\centering
\includegraphics[width=0.75\textwidth]{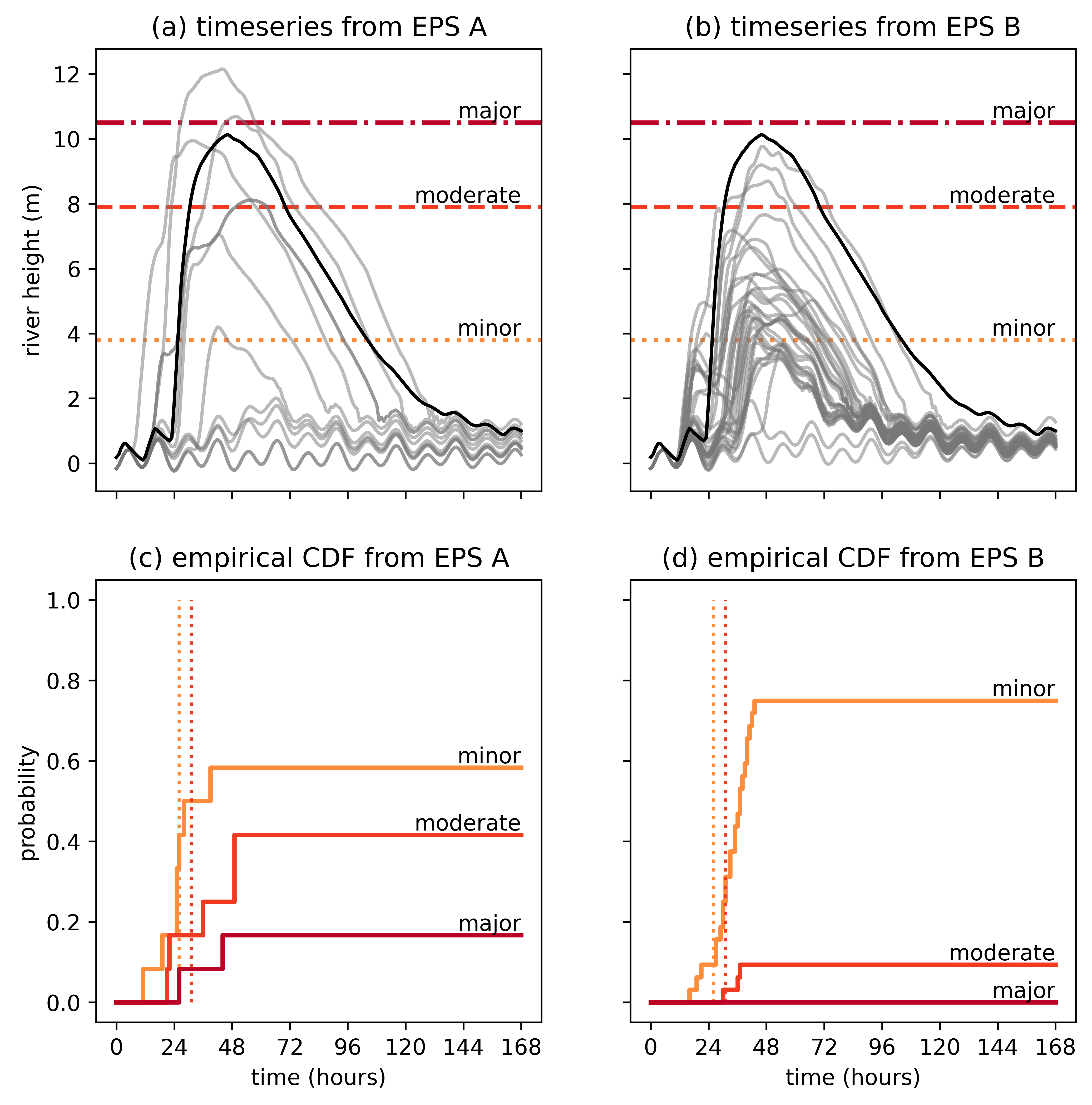}
\caption{River height forecasts (gray) for North Richmond from (a) EPS~A and (b) EPS~B with observed heights (black lines). Horizontal lines mark thresholds for minor, moderate, and major flooding. Empirical CDFs for first passage time for (c) EPS~A and (d) EPS~B for the minor, moderate and major flood thresholds. Vertical dotted lines indicate realized first passage times for minor and moderate thresholds; the major threshold was not exceeded. Time~0 corresponds to model initialization at 00:00~UTC on 6~June~2024.}
\label{fig:hnv hydrograph}
\end{figure}

For each EPS, time-series forecasts are converted to a partial empirical CDF $F$ for the first time to exceed height $h$, where $h$ is the minor, moderate, or major flood threshold. The CDF $F$ is partial because the probability $F(s)$ that the first passage time does not exceed time $s$ is only known when $s \leq 168$ hours. Figure~\ref{fig:hnv hydrograph}c,d shows these CDFs for North Richmond.

Right-censoring is prominent: in over 81\% of the 665 issuance times, the realized first passage time for the minor threshold exceeded the prediction horizon. In more than 87\% of cases, a complete empirical CDF from EPS~A could not be constructed because some ensemble members remained below the minor threshold throughout the horizon. Statistics for EPS~B are similar.

First passage time forecasts, interpreted as distributions, can be evaluated using the threshold-weighted CRPS (\ref{eq:twcrps chaining specific}). Using ensemble member values to approximate expectations in (\ref{eq:twcrps chaining specific}) gives
\begin{equation}\label{eq:twcrps ensemble}
\twCRPS_\tau(F,t) \approx \frac{1}{m}\sum_{i=1}^m \left|[x_i]_\tau-[t]_\tau\right| - \frac{1}{2m c_m}\sum_{i=1}^m\sum_{j=1}^m \left|[x_i]_\tau - [x_j]_\tau\right|,
\end{equation}
where $x_i$ is the time-to-event forecast of the $i$th ensemble member and $c_m$ equals either $m$ or $m-1$. If ensemble values are viewed as a random sample from an unspecified predictive distribution $F$, then the first term in (\ref{eq:twcrps ensemble}) is an unbiased estimator of the first expectation in (\ref{eq:twcrps chaining specific}), and the second term with $c_m=m-1$ is unbiased for the second expectation \citep{ferro2014fair}. Alternatively, if $F$ is interpreted as the empirical CDF of ensemble values, setting $c_m=m$ makes (\ref{eq:twcrps ensemble}) exact. In both cases, only knowledge of right-censored forecast values and the right-censored realization are required for score evaluation.

Table~\ref{tab:twcrps hydro results} shows mean threshold-weighted CRPS for EPS~A and EPS~B across all forecast cases at each flood threshold, using $\tau=168$ and $c_m=m-1$. EPS~A performed best on average. To assess significance, we compute for each threshold, system, and issuance time the mean score across eight locations (to reduce spatial dependence), yielding 665 score differentials between systems. The null hypothesis of equal predictive performance is rejected at the 5\% level in favor of EPS~A using the \cite{hering2011comparing} test statistic, which accounts for autocorrelation and is asymptotically standard normal under the null hypothesis \citep{diebold2002comparing}. Corresponding p-values appear in Table~\ref{tab:twcrps hydro results}.

\begin{table}[t!]
    \renewcommand{\arraystretch}{1.2} 
    \centering  \scriptsize
    \begin{tabular}{|c|c|r|r|r|r|r|}
    \hline
        flood threshold & mean $\twCRPS_\tau$ for EPS A & mean $\twCRPS_\tau$ for EPS B & p-value  \\ \hline
        minor & $3.291$ & $3.736$ & $0.031$ \\ 
        moderate & $0.624$ & $0.820$ & $0.020$  \\ 
        major & $0.001$ & $0.009$ & $0.008$ \\ \hline
    \end{tabular}
    \caption{Mean $\twCRPS_\tau$ (hours), where $\tau=168$, for EPS~A and EPS~B first passage time forecasts at minor, moderate and major flood thresholds, for issuances spanning 9~May 2024 to 31~July 2025. Lower scores indicate better performance. Each p-value is for the test of equal predictive performance versus EPS~A superiority.}
     \label{tab:twcrps hydro results}
\end{table}

Quantile and IQI forecasts can also be compared, with right-censored quantiles computed empirically from ensemble values. For example, IQI first passage time forecasts are scored using $\twIS_{0.5,168}$ (\ref{eq:twIS2}).  For the minor threshold, mean scores are $3.771$ hours for EPS~A and $4.014$ hours for EPS~B; the difference is not significant (p-value $0.171$).

\subsection{Diagnosing conditional miscalibration and recalibrating time-to-event quantile forecasts}\label{ss:calibration}

This example illustrates the use of isotonic regression as a diagnostic for conditional quantile miscalibration \citep{gneiting2023regression, gneiting2023model} in time-to-event forecasts and as a tool for forecast recalibration. Unlike the settings considered by \citet{gneiting2023regression}, both forecasts and realizations are right-censored.

Consider a forecast service designed for small recreational boats and windsurfers in Botany Bay, Sydney, Australia. Wind speeds above 15 knots ($7.7\,\mathrm{ms}^{-1}$) are unsafe for small boats, whereas windsurfers typically require winds of at least this strength. The service provides forecasts of the first time that wind speed exceeds 15 knots during an 18‑hour validity period (18:00-12:00 UTC), which spans daylight and twilight conditions. Forecasts are issued daily at 06:00 UTC --- 12~hours before the validity period --- and expressed as an interquartile interval (IQI). Verification is performed using observations from the Kurnell weather station on Botany Bay.

We use hourly wind speed forecasts from the high-resolution ACCESS model for the Sydney domain \citep{rennie2022access}, initialized daily at 00:00 UTC with a 36-hour horizon, and archived by the Jive verification system \citep{loveday2024jive}. Figure~\ref{fig:wind timeseries} shows forecasts and observations for the validity period 3~January~2023 18:00 UTC to 4~January~2023 12:00 UTC. The realized time-to-event is 5.35 hours after the start of the validity period; the forecast time-to-event is 15.35 hours, obtained via linear interpolation.

\begin{figure}[t!]
\centering
\includegraphics[width=0.45\textwidth]{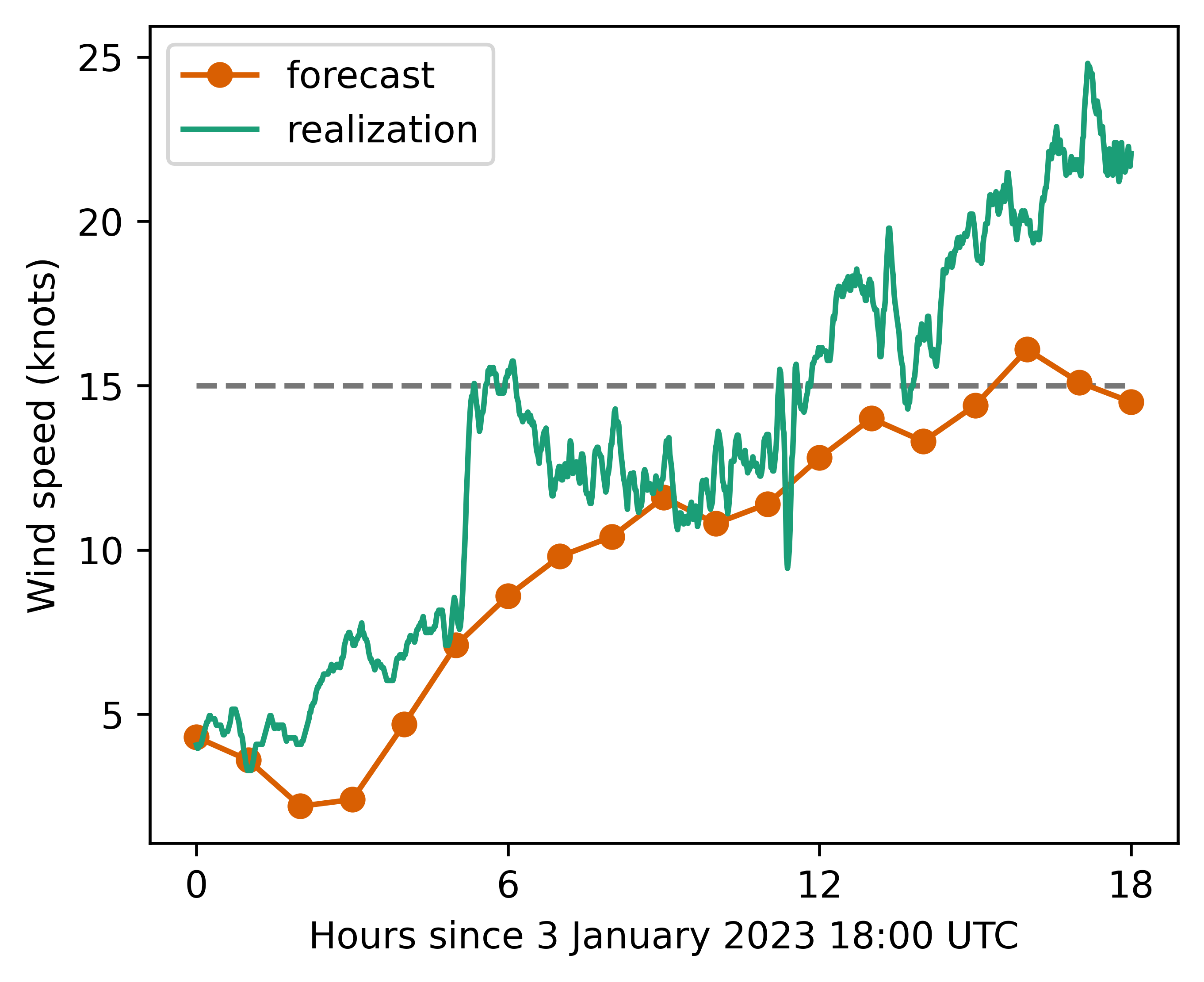}
\caption{Wind speed forecasts for Botany Bay from the ACCESS model (initialization 3~January~2023 00:00 UTC) and corresponding observations.}
\label{fig:wind timeseries}
\end{figure}

As an initial assessment, we derive single-valued time-to-event forecasts $x$ via linear interpolation from hourly ACCESS forecasts and interpret them as deterministic distributions, forming IQI forecasts $(x,x)$. Each validity period is 18 hours, so $\tau=18$, and scoring using $\twIS_{0.5,18}$ (\ref{eq:twIS2}) reduces to the absolute error between the right-censored forecast and realization. Using 2023 data, the mean score $\overline{\twIS}_{0.5,18}$ is 6.56 hours --- over one-third of the validity duration --- indicating substantial room for improvement.

\begin{figure}[t!]
\centering
\includegraphics[width=0.45\textwidth]{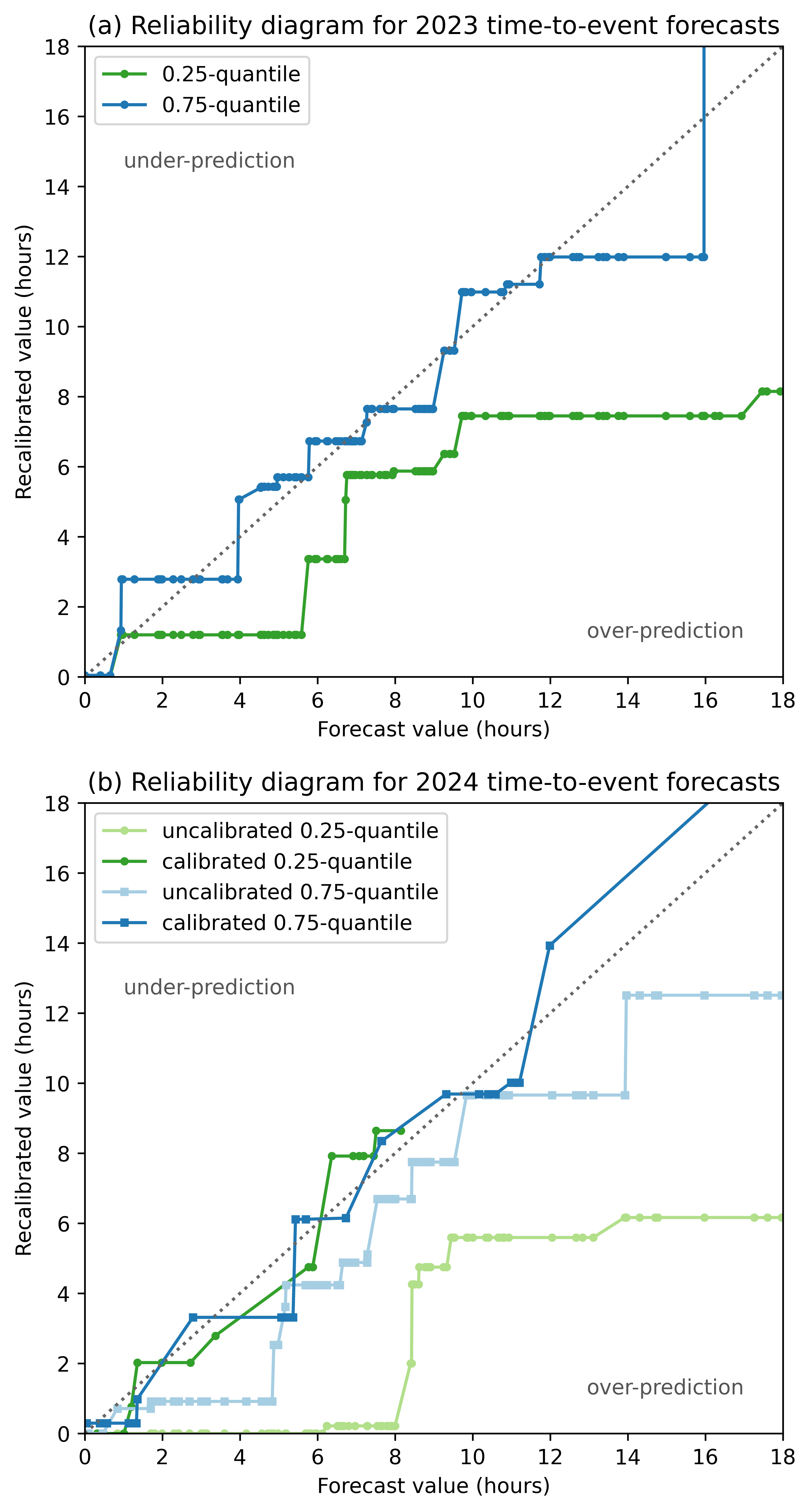}
\caption{Reliability curves from isotonic regression using Botany Bay data: (a) 2023 time-to-event forecasts derived from bias-corrected wind speed forecasts, (b) 2024 uncalibrated and calibrated time-to-event quantile forecasts.}
\label{fig:wind reliability}
\end{figure}

The aim is to produce improved IQI forecasts for the time at which wind speed first exceeds 15 knots, using a three-step process with 2023 data for training and 2024 data for evaluation.
\begin{enumerate}
\item Conditional-biases for hourly ACCESS wind speed forecasts in 2023 are calculated\footnote{Conditional biases are calculated using isotonic regression with squared loss as the loss function.}. For example, a forecast wind speed of 12 knots has an average underforecast bias of 3.1 knots. Apply these bias corrections to each hourly forecast. Then, for each daily 18-hour validity period (18:00--12:00 UTC) in 2023, the hours since 18:00 UTC that the bias-corrected wind speed forecast first exceeds 15 knots is computed via linear interpolation. If the threshold is not exceeded, assign a dummy value greater than 18 to the time-to-event forecast; we use 1000. As discussed later, the choice of dummy value is immaterial.
\item Regress the time-to-event forecasts from Step~1 against the 2023 realizations (using the dummy value 1000 when the event does not occur within the 18-hour window) via isotonic regression, minimizing the quantile losses $\QL_{0.25}$ and $\QL_{0.75}$ (\ref{eq:check loss}) under a monotonicity constraint. This yields two reliability curves (Figure~\ref{fig:wind reliability}a) for the 0.25 and 0.75-quantiles. For example, a time-to-event forecast of 8 hours derived from a bias-corrected hourly wind speed predictions converts to an IQI of $(5.875, 7.65)$ hours. The jump in the 0.75-quantile curve near 16 hours reflects forecasts assigned the dummy value, indicating that the optimal upper quantile lies beyond the 18-hour validity period.
\item Apply the same corrections to the 2024 forecasts: (i) bias-correct the ACCESS wind speed forecasts using the 2023 bias adjustments; (ii) derive each single-valued time-to-event forecast $x$ via linear interpolation; and (iii) convert each $x$ to an IQI $(x_1, x_2)$ using the reliability curves in Figure~\ref{fig:wind reliability}a.
\end{enumerate}

For 2024, the mean threshold-weighted interval score $\overline{\twIS}_{0.5,18}$ drops from 6.02 hours for uncalibrated forecasts to 2.66 hours for calibrated IQI predictions. Reliability curves for uncalibrated and calibrated quantile forecasts appear in Figure~\ref{fig:wind reliability}b; calibrated curves lie close to the diagonal, indicating good reliability. This improvement is feasible because sufficient events occurred in the 2023 training dataset: 73\% materialized, and bias-corrected wind speed forecasts predicted events in 54\% of cases. Rarer events would require larger training datasets.

We do not claim this recalibration method as best practice, but quantile isotonic regression has features well-suited to right-censored data. Time-to-event quantile forecasts for finite validity periods $\tau$ can be evaluated using threshold-weighted quantile scores, which depend only on right-censored forecasts and realizations, making the choice of dummy value irrelevant given it exceeds $\tau$. Moreover, \citet{jordan2022characterizing} show that the min--max solution to isotonic regression for quantiles is optimal for any consistent scoring function for quantiles, including threshold-weighted scores. This extends naturally to interval forecasts, assessed via the threshold-weighted interval score (a sum of two quantile scores). Consequently, the circle markers in Figure~\ref{fig:wind reliability} within the region $[0,\tau)\times[0,\tau)$ are unaffected by the choice of  dummy value. Choosing a dummy value greater than $\tau$ rather than equal to $\tau$ clarifies whether a quantile exceeds the validity period, aiding interpretation. For example, if the recalibrated 0.75-quantile exceeds $\tau$, one can state: ``the 75th percentile for the event time lies beyond the forecast validity period.''

\section{Practical implications}\label{s:discussion}

This paper has presented a framework for evaluating time-to-event forecasts when some forecasts, observations, or both are available only through their right-censored values relative to a common fixed censoring time $\tau$. The setting is particularly relevant for forecasting systems with finite prediction horizons or validity periods, where $\tau$ represents the forecast horizon or validity period for each forecast case. It does not address the random censoring mechanisms that dominate much of the biomedical survival-analysis literature.

For probabilistic forecasts, recent developments in scoring-rule theory provide theoretically sound evaluation methods. In particular, the localization framework of \citet{de2025localizing} implies that threshold-weighted variants of familiar scores, including the threshold-weighted CRPS (Equations~\ref{eq:twcrps specific} and \ref{eq:twcrps ensemble}) and threshold-weighted logarithmic score (Equation~\ref{eq:twlogs special case}), admit coherent evaluation under fixed-horizon right-censoring. These methods remain valid when forecasts themselves are right-censored to the prediction horizon. The threshold-weighted CRPS is particularly attractive for ensemble prediction systems in meteorology and hydrology because it can be computed directly from ensemble-derived empirical distributions without requiring the specification of a predictive density.

For single-valued and interval-valued forecasts, the choice of forecast summary is important. Quantiles and prediction intervals whose endpoints are quantiles can be evaluated in a theoretically sound manner under fixed-horizon right-censoring using the threshold-weighted quantile loss (Equation~\ref{eq:twCL}) and the threshold-weighted interval score (Equation~\ref{eq:twIS2}). In contrast, the mean (or expectation) functional does not admit a corresponding theoretically justified scoring method. As a result, mean-valued time-to-event forecasts, including those based on ensemble means, cannot be evaluated in the same principled manner as quantiles and prediction intervals when right-censoring is present.

Beyond forecast ranking, several familiar diagnostic tools remain applicable under fixed-horizon right-censoring. Threshold-weighted CRPS and quantile scores retain useful diagnostic representations through Brier-score decompositions and Murphy diagrams, respectively (Figure~\ref{fig:murphy synthetic}). Reliability diagrams based on quantile isotonic regression (Figure~\ref{fig:wind reliability}) can be constructed when both quantile forecasts and corresponding realizations are represented by their right-censored values, providing practical tools for diagnosing conditional biases and recalibrating forecasts.

The examples presented in this paper demonstrate that these methods can be implemented in operational forecasting environments. The open-source Python package \texttt{scores} \citep{leeuwenburg2025scores, leeuwenburg2024scores} contains implementations of the scoring, diagnostic and recalibration methods discussed in this paper, with the exception of the threshold-weighted logarithmic score. This substantially lowers the barrier to applying these methods in operational forecast verification workflows.

The key practical message is therefore straightforward: when time-to-event forecasts are subject to fixed-horizon right-censoring, predictive distributions, quantiles, and prediction intervals can be evaluated using coherent scoring and diagnostic tools, whereas the evaluation of expected time-to-event forecasts remains fundamentally problematic under fixed-horizon right-censoring.

\vspace{10pt}

\noindent\textbf{Acknowledgments.} The authors would like to thank Tilmann Gneiting, Deryn Griffith, Richard Laugesen and Jason West for giving feedback on an earlier version of this paper, and Chris Leahy for making data available for the Hawkesbury--Nepean Valley case study.

\vspace{10pt}

\noindent\textbf{Declaration of generative AI and AI-assisted technologies in the manuscript preparation process.} Copilot was employed in the final stages of this work to refine writing style and improve brevity. After using Copilot, the authors reviewed and edited the content as needed and take full responsibility for the content of the article.

\vspace{10pt}

\noindent\textbf{Data and code availability.} A repository of data and code for the case studies of this paper
has been made available \citep{taggart2026code}, and uses the open source python \texttt{scores} package \citep{leeuwenburg2025scores, leeuwenburg2024scores}.

\appendix

\section{Proofs of results}\label{a:proofs}

Theorem~\ref{thm:provisional from strictly proper 1} follows from Theorem~2 of \citet{de2025localizing}. However, the direct proof given below is shorter, avoids introducing the full localization framework, and provides clear intuition for the result.

\begin{proof}[Proof of Theorem~\ref{thm:provisional from strictly proper 1}]
Suppose that $\mathcal{F}$, $\mathcal{F}_*$ and $S$ satisfy the hypotheses of Theorem~\ref{thm:provisional from strictly proper 1}. Given any $F$ and $G$ in $\mathcal{F}$, the expectation $\EEE_{[F]_\tau} S([G]_\tau, T)$ exists, since $[F]_\tau, [G]_\tau \in\mathcal{F}_*$ and $S$ is strictly proper relative to $\mathcal{F}_*$. Moreover, since $[T]_\tau$ (when $T\sim F$) and $T$ (when $T\sim [F]_\tau$) have the same distribution, a simple calculation shows that
\begin{align*}
\EEE_{[F]_\tau} S([G]_\tau, T)
&= \EEE_F S([G]_\tau, [T]_\tau) \\
&= \EEE_F S_\tau(G, [T]_\tau),
\end{align*}
where the last equality follows from the definition of $S_\tau$. Consequently, for any $F$ and $G$ in $\mathcal{F}$,
\begin{align*}
\EEE_F S_\tau(F,[T]_\tau)
&= \EEE_{[F]_\tau} S([F]_\tau, T) \\
&\leq \EEE_{[F]_\tau} S([G]_\tau, T) \\
&= \EEE_{F} S_\tau(G, [T]_\tau)
\end{align*}
by the strict propriety of $S$ relative to $\mathcal{F}_*$, with equality if and only if $[F]_\tau=[G]_\tau$. Hence $S_\tau$ is strictly locally proper relative to $\mathcal{F}$.
\end{proof}

Corollary~\ref{cor:twCRPS} likewise follows as a special case of the framework of \citet{de2025localizing}. Nonetheless, we give a direct proof of part~(i) and a proof of part~(ii) in the special case $w(s)=\one\{0\leq s < \tau\}$. These proofs provide insight into how Theorem~\ref{thm:provisional from strictly proper 1} may be applied to both a distance-sensitive scoring rule ($\CRPS$) and a semi-local scoring rule ($\LogS$). The extension of part~(ii) to general weight functions requires more substantial measure-theoretic machinery, and we therefore refer interested readers to \citet{de2025localizing}.

\begin{proof}[Proof of Corollary~\ref{cor:twCRPS}]
Suppose that $w:I\to[0,1]$ is positive on $[0,\tau)$ and zero elsewhere. 

To prove part (i), define $w_1$ by
\[
w_1(s)=
\begin{cases}
w(s) & \text{if } 0 \leq s \leq \tau \\
w(\tau) & \text{if } s > \tau.
\end{cases}
\]
Then the chaining function $v_1$, defined by $v_1(x)=\int_{[0,x)}w_1(s)\,\dd s$, is injective, whereby the scoring function $S=\twCRPS(\,\cdot\,,\,\cdot\,;w_1)$ is strictly proper relative to the class of distributions on $I$ with finite first moment \citep[Proposition~3]{allen2023evaluating}. We now apply Theorem~\ref{thm:provisional from strictly proper 1} to the strictly proper scoring rule $S$ to obtain the strictly locally proper scoring rule $S_\tau$ relative to the class of distributions on $I$, noting that  $S_\tau=\twCRPS(\,\cdot\,,\,\cdot\,;w)$.

We turn now to part (ii) for the case $w(s)=\one\{0\leq s < \tau\}$. The key insight is to apply the logarithmic score to densities with respect to a measure that is compatible with right-censoring. Let $\mathcal{F}$ denote the class of distributions $F$ possessing a Lebesgue density $f$. Define the measure
\[
\nu := \lambda\vert_{[0,\tau)} + \delta_\tau,
\]
where  $\lambda\vert_{[0,\tau)}$ denotes the restriction of Lebesgue measure $\lambda$ to $[0,\tau)$ and $\delta_\tau$ denotes the Dirac measure at $\tau$. The measure $\nu$ is $\sigma$-finite and every censored distribution $[F]_\tau$, where $F\in\mathcal{F}$, is absolutely continuous with respect to $\nu$. Hence the class $\mathcal{F}_*:=\{[F]_\tau: F\in\mathcal{F}\}$ is contained in the class $\mathcal{L}_\nu$ of probability distributions that are absolutely continuous with respect to $\nu$.

Now let $S$ define the logarithmic score relative to $\nu$-densities, so that
\[
S(P, t) = -\ln \frac{\dd P}{\dd \nu}(t)
\]
for $P$ in $\mathcal{L}_\nu$. By Section~4.1 of \citet{gneiting2007strictly}, $S$ is strictly proper relative to $\mathcal{L}_\nu$. Therefore Theorem~\ref{thm:provisional from strictly proper 1} implies that the score $S_\tau$, given by $S_\tau(F,t)=S([F]_\tau,t)$, is strictly locally proper relative to $\mathcal{F}$ on $[0,\tau)$.

It remains to show that $S_\tau=\twLogS_\tau$. Now if $F\in\mathcal{F}$ then
\[
\frac{\dd[F]_\tau}{\dd \nu}(s) = \begin{cases}f(s), & 0\leq s < \tau, \\ 1 - F(\tau), & s=\tau, \end{cases}
\]
$\nu$-almost everywhere (c.f. \citealt[Section~B.1]{de2025localizing}). Substituting this into the formula for $S_\tau$ yields precisely $\twLogS_\tau$.
\end{proof}

We now prove the results about local elicitability, starting with two lemmas.

\begin{lem}\label{lem:censor}
Suppose that $F$ is a probability distribution on $I$, $S$ is a scoring function and $\tau$ belongs to the interior of $I$. Then
\begin{equation}\label{eq:censor lemma}
\EEE_F S(x, [T]_\tau) = \EEE_{[F]_\tau}S(x,T) 
\end{equation}
for all $x$ in $I$.
\end{lem}

\begin{proof}
The definition of the expectation, right-censored distributions and right-censored random variables imply that both sides of Equation~(\ref{eq:censor lemma}) equal $\int_{[0,\tau)}S(x,t)\,\dd F(t) + S(x,\tau)(1-F(\tau))$.
\end{proof}

Suppose that $0<\alpha<1$. Recall that a number $x$ is an $\alpha$-quantile of a distribution $F$ if and only if $\lim_{y\uparrow x} F(y) \leq \alpha \leq F(x)$.

\begin{lem}\label{lem:quantile censoring}
Let $U$ denote the $\alpha$-quantile functional for some class $\mathcal{F}$ of probability distributions on $I$. Suppose that $\tau$ belongs to the interior of $I$ and $F\in\mathcal{F}$.
\begin{enumerate}
\item[(a)] If $x\in U(F)$ then $[x]_\tau \in U([F]_\tau)$.
\item[(b)] $U([F]_\tau)$ is a subset of $[U(F)]_\tau$.
\end{enumerate}
\end{lem}

\begin{proof}
To prove (a), suppose that $x\in U(F)$. If $[x]_\tau = x < \tau$ then
\[\lim_{y\uparrow [x]_\tau} [F]_\tau(y) = \lim_{y\uparrow x} F(y) \leq \alpha \leq F(x) = [F]_\tau([x]_\tau).\]
On the other hand, if $[x]_\tau = \tau$ then
\[\lim_{y\uparrow [x]_\tau} [F]_\tau(y) = \lim_{y\uparrow \tau} F(y) \leq \lim_{y\uparrow x} F(y) \leq \alpha \leq 1 =  [F]_\tau([x]_\tau).\]
Either way, $[x]_\tau \in U([F]_\tau)$.

To prove (b), suppose that $x\in U([F]_\tau)$. Then necessarily $x\leq \tau$ and for any $0\leq y<\tau$ we have $[F]_\tau(y)=F(y)$. Hence
\[ \lim_{y\uparrow x}F(y) = \lim_{y\uparrow x}[F]_\tau(y)\leq \alpha.\]
Now if $x<\tau$ we also have $\alpha \leq [F]_\tau(x)=F(x)$ and hence $x$ is an $\alpha$-quantile of $F$, which combined with $x<\tau$ implies that $x\in[U(F)]_\tau$. If, on the other hand, $x=\tau$ then write $\beta=\lim_{y\uparrow \tau}F(y)$. It is clear that $\beta \leq \alpha$ and that $\tau$ is a $\beta$-quantile of $F$. Hence there must be an $\alpha$-quantile $q$ of $F$ such that $q\geq\tau$. Now $[q]_\tau= x$ whence $x\in[U(F)]_\tau$.
\end{proof}

\begin{proof}[Proof of Theorem~\ref{thm:quantile is provisionally elicitable}]
To prove (a), suppose that $\tau$ belongs to the interior of $I$ and define the scoring function $S: I \times I\to\RRR$ by Equation~(\ref{eq:consistent provisional quantile sf}), where $g:I\to\RRR$ is strictly increasing on $[0,\tau]$. Suppose that $F\in\mathcal{F}$, $U$ denotes the $\alpha$-quantile functional on $\mathcal{F}$, $u\in U(F)$ and $x\in I$. By Lemma~\ref{lem:quantile censoring}(a), $[u]_\tau$ is an $\alpha$-quantile of $[F]_\tau$. Hence
\begin{align*}
\EEE_F S(u,[T]_\tau)
&= \EEE_{[F]_\tau} S(u,T) \\
&= \EEE_{[F]_\tau} S([u]_\tau,T) \\
&\leq  \EEE_{[F]_\tau} S([x]_\tau,T) \\
&= \EEE_{[F]_\tau} S(x,T) \\
&= \EEE_F S(x,[T]_\tau),
\end{align*}
where the first and last equalities are justified by Lemma~\ref{lem:censor}, the second and third equalities are justified by the definition of $S$, and the inequality is justified by the fact that $S$, restricted to the domain $[0,\tau]\times[0,\tau]$, is a strictly consistent scoring function for the class $\mathcal{F}_0$ of probability measures supported on the compact interval $[0,\tau]$ \citep{gneiting2011quantiles}. This establishes (\ref{eq:censored consistency}). Moreover, if equality holds in the above calculation, then the strict consistency of $S$ when restricted to the compact interval $[0,\tau]$ implies that $[x]_\tau$ is an $\alpha$-quantile of $[F]_\tau$, which implies by Lemma~\ref{lem:quantile censoring}(b) that $[x]_\tau\in [U(F)]_\tau$. Hence $S$ is strictly locally consistent. Part (b) of the theorem follows from part (a) and the fact that the $\alpha$-quantile functional is elicitable with respect to $\mathcal{F}$ \citep[Theorem~9]{gneiting2011making}.
\end{proof}

\begin{lem}\label{lem:vector valued functional}
Suppose that, whenever $1\leq i\leq k$, the scoring function $S_i:I\times I\to\RRR$ is strictly locally consistent for the functional $U_i:\mathcal{F}\to\mathcal{P}(I)$ relative to $\mathcal{F}$ on $[0,\tau)$, and that $\lambda_i>0$. Let $U:\mathcal{F}\to\mathcal{P}(I^k)$ denote the functional given by $U=(U_1,\ldots,U_k)$. Then $U$ is locally elicitable relative to $\mathcal{F}$ on $[0,\tau)$ and the scoring function $S:I^k\times I \to \RRR$, given by
\[S((x_1, \ldots, x_k), t) = \sum_{i=1}^k \lambda_i S_i(x_i,t),\]
is strictly locally consistent for the functional $U$ relative to $\mathcal{F}$ on $[0,\tau)$.
\end{lem}

\begin{proof}
Assume the hypotheses of the lemma. Suppose that $u_i\in U_i$ and $x_i\in I$ whenever $1\leq i\leq k$. Then
\begin{align*}
\EEE_F S((u_1, \ldots, u_k), [T]_\tau)
&=  \sum_{i=1}^k \lambda_i \EEE_F S_i(u_i, [T]_\tau) \\
&\leq  \sum_{i=1}^k \lambda_i \EEE_F S_i(x_i, [T]_\tau) \\
&= \EEE_F S((x_1, \ldots, x_k), [T]_\tau)
\end{align*}
by the strict local consistency of each $S_i$ and the linearity of the expectation operator. If equality holds then we must have $\EEE_F S_i(u_i, [T]_\tau) = \EEE_F S_i(x_i, [T]_\tau)$ for each $i$, whence $[x_i]_\tau \in [U_i(F)]_\tau$ for each $i$. Hence $[(x_1,\ldots,x_k)]_\tau \in [U(F)]_\tau$.
\end{proof}

\begin{proof}[Proof of Theorem~\ref{thm:prediction interval is provisionally elicitable}]
Note that $\twIS_{\alpha,\tau}((x_1,x_2), t)=\twQL_{\alpha/2,\tau}(x_1,t)+\twQL_{1-\alpha/2,\tau}(x_2,t)$. The theorem then follows from Theorem~\ref{thm:quantile is provisionally elicitable} and Lemma~\ref{lem:vector valued functional}.
\end{proof}

\begin{proof}[Proof of Theorem~\ref{thm:not prov elic}]
Suppose that properties (a), (b) and (c) of the proposition hold. Assume that $U$ is locally elicitable with respect to $\mathcal{F}$ on $[0,\tau)$. Then, by Definition~\ref{def:provisionally elicitable} and property (a), there exists a scoring function $S: I\times I\to \RRR$ such that, for each $i=1,2$, the set of minimizers of the mapping $x\mapsto \EEE_{F_i}S(x,[T]_\tau)$ is precisely $U(F_i)$. By Lemma~\ref{lem:censor} and property (c),
\begin{align*}
\EEE_{F_1}S(x,[T]_\tau)
&=\EEE_{[F_1]_\tau}S(x,T) \\
&=\EEE_{[F_2]_\tau}S(x,T) \\
&=\EEE_{F_2}S(x,[T]_\tau).
\end{align*}
Hence the set of minimizers of $x\mapsto \EEE_{F_1}S(x,[T]_\tau)$ is both $U(F_1)$ and $U(F_2)$, which contradicts property (b). Therefore $U$ is not locally elicitable.
\end{proof}

\begin{proof}[Proof of Corollary~\ref{cor:expectation not provisionally elicitable}]
Denote by $U$ the expectation functional and suppose that $\tau$ lies in the interior of $I$.
Choose $a$ in the interior of $I$ such that $\tau<a<2\tau$. Let $f_1$ and $f_2$ denote the probability densities given by
\begin{align*}
f_1(s)&=
\begin{cases}
1/a & \text{if } 0\leq s \leq a \\
0 & \text{otherwise},
\end{cases} \\
f_2(s)&=
\begin{cases}
1/a & \text{if } 0\leq s \leq \tau \\
\frac{2}{a}(1 - (s-\tau)/(a-\tau)) & \text{if } \tau < s \leq a \\
0 & \text{otherwise,}
\end{cases}
\end{align*}
and let $F_1$ and $F_2$ denote their respective CDFs. Note that $F_1$ and $F_2$ have finite support, density functions and finite $m$th moment. Now $0 < U(F_2) < U(F_1)=a/2<\tau$, whence $F_1$ and $F_2$ satisfy properties (a), (b) and (c) of Theorem~\ref{thm:not prov elic}.  Hence $U$ is not locally elicitable with respect to $\mathcal{F}$ on $[0,\tau)$.
\end{proof}

\bibliography{time_of_event.bib}
\bibliographystyle{apalike-ejor}

\end{document}